\documentclass[11pt]{amsart} 
\usepackage{amsmath} \usepackage{amsxtra}
\usepackage{amscd} \usepackage{amsthm}  
\usepackage{amsfonts}\usepackage{amssymb} 
\usepackage{eucal}\usepackage{bm}
\usepackage{graphicx} 
\usepackage[latin1]{inputenc}
\newtheorem{Lm}[subsubsection]{Lemma}
\newtheorem{Pp}[subsubsection]{Proposition}
\newtheorem{Con}[subsubsection]{Conjecture}
\newtheorem{Thm}[subsubsection]{Theorem}
\newtheorem{Def}[subsubsection]{Definition}
\newtheorem{Rem}[subsubsection]{Remark}

\theoremstyle{definition}

\theoremstyle{remark}

\newcommand{\nc}{\newcommand}
\nc{\renc}{\renewcommand}
\nc{\ssec}{\subsection}
\nc{\sssec}{\subsubsection}
\nc{\on}{\operatorname}

\nc\ol{\overline}
\nc\wt{\widetilde}
\nc\tboxtimes{\wt{\boxtimes}}

\newcommand{\cA}{{\mathcal A}}

\newcommand{\cD}{{\mathcal D}}
\newcommand{\cH}{{\mathcal H}}
\newcommand{\cE}{{\mathcal E}}
\newcommand{\cG}{{\mathcal G}}
\newcommand{\cI}{{\mathcal I}}

\newcommand{\cO}{{\mathcal O}}
\newcommand{\cL}{{\mathcal L}}

\newcommand{\cN}{{\mathcal N}}
\newcommand{\cF}{{\mathcal F}}
\newcommand{\cK}{{\mathcal K}}
\newcommand{\cP}{{\mathcal P}}

\newcommand{\cS}{{\mathcal S}}
\newcommand{\cT}{{\mathcal T}}
\newcommand{\cU}{{\mathcal U}}
\newcommand{\cV}{{\mathcal V}}
\newcommand{\cW}{{\mathcal W}}
\newcommand{\cX}{{\mathcal X}}
\newcommand{\cY}{{\mathcal Y}}
\newcommand{\cZ}{{\mathcal Z}}

\renewcommand{\AA}{{\mathbb A}}

\newcommand{\ZZ}{{\mathbb Z}}

\newcommand{\PP}{{\mathbb P}}

\newcommand{\OO}{{\mathbb O}}

\nc{\gi}{\mathfrak{i}}
\nc{\ga}{\mathfrak{a}}
\nc{\gb}{\mathfrak{b}}
\renewcommand{\gg}{\mathfrak{g}}  
\newcommand{\gm}{\mathfrak{m}}    

\newcommand{\gp}{\mathfrak{p}}
\newcommand{\gq}{\mathfrak{q}}
\newcommand{\gu}{\mathfrak{u}}

\newcommand{\gv}{\mathfrak{v}}

\nc{\gL}{\mathfrak{L}}
\nc{\gM}{\mathfrak{M}}
\nc{\gV}{\mathfrak{V}}
\nc{\gE}{\mathfrak{E}}
\nc{\gJ}{\mathfrak{J}}
\nc{\gF}{\mathfrak{F}}
\nc{\gN}{\mathfrak{N}}
\nc{\gsl}{\mathfrak{sl}} 

\nc{\bA}{\mathbf{A}}
\nc{\bC}{\mathbf{C}}

\nc{\uZ}{\underline{\cZ}}

\nc{\MAPS}{{\mathcal Maps}}

\newcommand{\Rep}{{\on{Rep}}}

\newcommand{\Gm}{\mathbb{G}_m}
\newcommand{\Ga}{\mathbb{G}_a}
\newcommand{\A}{\mathbb{A}}

\newcommand{\toup}[1]{\stackrel{#1}{\to}}

\newcommand{\hook}[1]{\stackrel{#1}{\hookrightarrow}}

\newcommand{\getsup}[1]{\stackrel{#1}{\gets}}

\newcommand{\Sp}{\on{\mathbb{S}p}}

\newcommand{\GSp}{\on{G\mathbb{S}p}}
\newcommand{\IC}{\on{IC}}

\newcommand{\Hom}{\on{Hom}}

\newcommand{\RG}{\on{R\Gamma}}

\newcommand{\Bun}{\on{Bun}}

\newcommand{\Bunb}{\on{\overline{Bun}} }
\newcommand{\Bunt}{\on{\widetilde\Bun}}

\newcommand{\Spec}{\on{Spec}}

\newcommand{\Gr}{\on{Gr}}

\newcommand{\GL}{\on{GL}}

\newcommand{\Eis}{{\on{Eis}}}
\newcommand{\pr}{\on{pr}}
\newcommand{\id}{\on{id}}

\newcommand{\QED}{$\square$} 
  
\newcommand{\Fp}{\mathbb{F}_p}  

\newcommand{\iso}{{\widetilde\to}}

\newcommand{\comp}{\circ}
\newcommand{\Four}{\on{Four}}
\renewcommand{\H}{{\on{H}}}   

\newcommand{\R}{\on{R}\!}   

\newcommand{\DD}{\mathbb{D}}  
\newcommand{\ov}[1]{\overline{#1}}
\newcommand{\select}[1]{{\it{#1}}}

\newcommand{\und}[1]{\underline{#1}}

\newcommand{\<}{\langle}
\renewcommand{\>}{\rangle}

\newcommand{\ev}{\mathit{ev}}

\newcommand{\Lie}{\on{Lie}}

\newcommand{\Res}{\on{Res}}

\newcommand{\act}{\on{act}}
\newcommand{\dimrel}{\on{dim.rel}}

\newcommand{\SL}{\on{SL}}

\newcommand{\Ind}{\on{Ind}}

\newcommand{\LocSys}{\on{LocSys}}

\newcommand{\ra}{\rightarrow}
\newcommand{\la}{\leftarrow}

\nc{\Perv}{\on{Perv}}

\nc{\Gra}{\on{Gra}}
\nc{\PPerv}{\on{{\PP}erv}}

\nc{\oX}{\overset{\scriptscriptstyle\circ}{X}}
\nc{\ocL}{\overset{\scriptscriptstyle\circ}{\cL}}
\nc{\gRes}{\on{gRes}}
\nc{\Sign}{\on{Sign}}
\nc{\goodat}{\rm{good\, at}}
\nc{\Whit}{\on{Whit}}
\nc{\add}{\on{add}}
\nc{\FS}{\on{FS}}
\nc{\oo}[1]{\overset{\scriptscriptstyle\circ}{#1}}
\nc{\can}{\on{can}}
\nc{\summ}{\on{sum}}
\nc{\SiSu}{\on{SS}}
\nc{\Irr}{\on{Irr}}
\nc{\Hecke}{\on{Hecke}}
\nc{\oHecke}{\overset{\scriptstyle\bullet}{\Hecke}}
\nc{\ocO}{\overset{\scriptstyle\bullet}{\cO}}
\nc{\og}[1]{\overset{\scriptscriptstyle\bullet}{#1}}
\nc{\of}{\overset{\scriptstyle\bullet}{f}}
\nc{\Exp}{\on{{\mathcal E}xp}}
\nc{\Chain}{\on{Chain}}
\nc{\Map}{\on{Map}}
\nc{\cSet}{\on{{\mathcal S}et}}
\nc{\Cat}{\on{\mathcal{C}at}}
\nc{\bfitDelta}{\bm{\mathit{\Delta}}}
\nc{\Grpd}{\on{Grpd}}
\nc{\Kan}{\on{{\mathcal K}an}}
\nc{\Spc}{\on{Spc}}
\nc{\Yon}{\on{Yon}}
\nc{\colim}{\on{colim}}
\nc{\Fin}{\on{{\mathcal F}in}}
\nc{\Alg}{\on{Alg}}
\nc{\Triv}{\on{\mathcal Triv}}
\nc{\Grp}{\on{{\mathcal G}rp}}
\nc{\EM}{\on{{\mathcal EM}}}
\nc{\Surj}{\on{Surj}}
\nc{\Ass}{\on{{\mathcal A}ss}}
\nc{\Sptr}{\on{Sptr}}
\nc{\cPr}{\on{{\cP}r}}
\nc{\Grd}{\on{Grd}}
\nc{\CAT}{\on{\bf 1-Cat}}
\nc{\CATT}{\on{\bf 2-Cat}}
\nc{\DGCat}{\on{DGCat}}
\nc{\Act}{\on{Act}}
\nc{\Env}{\on{Env}}
\nc{\Quad}{\on{Quad}}
\nc{\ComGrp}{\on{ComGrp}}
\nc{\PreStk}{\on{PreStk}}
\nc{\Stk}{\on{Stk}}
\nc{\NearStk}{\on{NearStk}}
\nc{\Tot}{\on{Tot}}
\nc{\Ptd}{\on{Ptd}}
\nc{\Mon}{\on{Mon}}
\nc{\Idem}{\on{Idem}}
\nc{\ind}{\on{ind}}
\nc{\BMod}{\on{BMod}}
\nc{\Tens}{\on{Tens}}
\nc{\Step}{\on{Step}}
\nc{\MAP}{\on{\bf Map}}
\nc{\Seq}{\on{Seq}}
\nc{\boneCat}{\on{\mathbf{1-Cat}}}
\nc{\DG}{\on{DG}}
\nc{\cAb}{\on{{\cA}b}}
\nc{\QCoh}{\on{QCoh}}
\nc{\otimesshriek}{\stackrel{!}{\otimes}}
\nc{\Ran}{\on{Ran}}
\nc{\Groth}{\on{Groth}}
\nc{\coGroth}{\on{coGroth}}
\nc{\Conf}{\on{Conf}}
\nc{\Tr}{\on{Tr}}
\nc{\FactMod}{\on{FactMod}}
\nc{\e}{\mathrm e}
\nc{\oblv}{\on{oblv}}
\nc{\Fib}{\on{Fib}}
\nc{\cofib}{\on{cofib}}
\nc{\PsId}{\on{Ps-Id}}
\nc{\Vac}{\on{Vac}}
\nc{\Reg}{\on{Reg}}
\nc{\bb}[1]{\overset{\scriptscriptstyle\bullet}{#1}}
\nc{\coInd}{\on{coInd}}
\nc{\FactGe}{\on{FactGe}}
\nc{\SI}{\on{SI}}
\nc{\coind}{\on{coind}}
\nc{\Sat}{\on{Sat}}
\nc{\QLisse}{\on{QLisse}}
\nc{\IndLisse}{\on{IndLisse}}
\nc{\ucoHom}{\on{\und{coHom}}}
\nc{\RMod}{\on{RMod}}
\nc{\Lisse}{\on{Lisse}}
\nc{\Mir}{\on{Mir}}
\nc{\psu}{\on{ps-u}}
\nc{\Maps}{\on{Maps}}
\nc{\IndCoh}{\on{IndCoh}}
\nc{\DMod}{\on{\cD-Mod}}
\nc{\Nilp}{\on{Nilp}}
\nc{\Arth}{\on{Arth}}

\makeatletter
\newcommand*{\doublerightarrow}[2]{\mathrel{
  \settowidth{\@tempdima}{$\scriptstyle#1$}
  \settowidth{\@tempdimb}{$\scriptstyle#2$}
  \ifdim\@tempdimb>\@tempdima \@tempdima=\@tempdimb\fi
  \mathop{\vcenter{
    \offinterlineskip\ialign{\hbox to\dimexpr\@tempdima+1em{##}\cr
    \rightarrowfill\cr\noalign{\kern.5ex}
    \rightarrowfill\cr}}}\limits^{\!#1}_{\!#2}}}
\newcommand*{\triplerightarrow}[1]{\mathrel{
  \settowidth{\@tempdima}{$\scriptstyle#1$}
  \mathop{\vcenter{
    \offinterlineskip\ialign{\hbox to\dimexpr\@tempdima+1em{##}\cr
    \rightarrowfill\cr\noalign{\kern.5ex}
    \rightarrowfill\cr\noalign{\kern.5ex}
    \rightarrowfill\cr}}}\limits^{\!#1}}}
\makeatother

\begin{document}
\title{Fourier coefficients and a filtration on $Shv(\Bun_G)$}
\author{S. Lysenko}
\address{Institut Elie Cartan Lorraine, Universit\'e de Lorraine, 
 B.P. 239, F-54506 Vandoeuvre-l\`es-Nancy Cedex, France}
\email{Sergey.Lysenko@univ-lorraine.fr}
\begin{abstract}
We define a filtration by $\DG$-subcategories on the $\DG$-category $Shv(\Bun_G)$ of sheaves on the moduli of $G$-torsors on a curve, which is stable under the action of  Hecke functors. We formulate a conjecture relating this filtration with another filtration on the spectral side of the categorical geometric Langlands conjecture. We also formulate a conjectural compatibility with the parabolic induction.
\end{abstract}
\maketitle

\section{Introduction} 
\sssec{} The idea to associate with a nilpotent orbit in $\gg$ a set of Fourier coefficients of automorphic forms on $G$ has appeared first maybe in \cite{GRS2}. More general conjectures at the level of functions were proposed in \cite{Ginz}. Further works in this direction in the theory of automorphic forms include, in particular, \cite{J1, JL, JL2, JLS}. 
A conceptual formulation appears, in particular, as (\cite{J1}, Conjectures~4.2 and 4.3).  
 
  Connections between nilpotent orbits in $\gg$ and representations of finite reductive groups also were used in preceding works of Lusztig on the classification of such representations \cite{Lus}.

\sssec{} We propose a way to fit the above ideas and, in particular, (\cite{J1}, Conjectures~4.2 and 4.3) into the global nonramified geometric Langlands program as it is formulated in \cite{AG}, \cite{AGKRRV}. Our proposal makes sense in several contexts.

 In the context of $\cD$-modules we propose some filtrations on the $\DG$-categories $\DMod(\Bun_G)$ and $\IndCoh_{\Nilp}(\LocSys(\check{G}))$, which are expected to correspond to each other under the conjectural geometric Langlands equivalence (\cite{AG}, Conjecture 1.1.6)
\begin{equation}
\label{conj_global_Langlands_for_D-mod}
\DMod(\Bun_G)\,\iso\, \IndCoh_{\Nilp}(\LocSys(\check{G}))
\end{equation}

 In the restricted constructible context we propose a similar filtrations on both sides and a similar refinement of (\cite{AGKRRV}, Conjecture 21.2.7). More precisely, in this context we propose a filtration on $Shv(\Bun_G)$ which is shown to be preserved by Hecke functors. The desired filtration on $Shv_{\Nilp}(\Bun_G)$ (also preserved by Hecke functors) is obtained by restriction from that on $Shv(\Bun_G)$.

\ssec{Conventions} 

\sssec{} Work in the constructible context over an algebraically closed field $k$ of characteristic $p\ge 0$ in the sense of (\cite{AGKRRV}, Section~1.1.1). We assume the charracteristic of $p$ \select{very good} in the sense of (\cite{AGKRRV}, Section~D.1.1), the precise definition is found in \cite{McN}. Since $k$ is algebraically closed, we systematically ignore the Tate twists. We use the notations and conventions for the constructible sheaves theory from \cite{AGKRRV}. In particular, $e$ denotes the field of coefficients of our sheaf theory. 

\sssec{} 
\label{Sect_1.1.2_on_conventions}
Let $G$ be reductive connected group over $k$, $X$ a smooth proper irreducible curve over $k$. In the case $p=0$ the nilpotent orbits in $\gg=\Lie G$ are classified in \cite{CM} (or \cite{Ra} for a short summary). Our assumption on $p$ implies that there is a $G$-invariant isomorphism $B: \gg\,\iso\,\gg^*$, which we fix. 

 We use the Fourier transform, in the case $p>0$ it is normalized `to preserve perversity as much as possible'. For this, we pick an injective character $\psi: \Fp\to e^*$ and denote by $\cL_{\psi}$ the corresponding Artin-Schreier sheaf $\cL_{\psi}$ on $\A^1$. In the case $p=0$ the Artin-Schreier sheaf does not exist and we always use the Kirillov model to get a similar definition of the Fourier transform as in (\cite{GL}, Appendix A).
 
 We let $\check{G}$ be the Langlands dual as a group over $e$. We also fix a $\check{G}$-invariant isomorphism $\check{\kappa}: \check{\gg}\,\iso\, \check{\gg}^*$.
  
  Write $\cX(\gg)$ (resp., $\cX(\check{\gg})$) for the set of nilpotent orbits in $\gg$ (resp., in $\check{\gg}$). Let $\cN(\gg)\subset\gg$ be the variety of nilpotent elements, and similarly for $\cN(\check{\gg})$. If a maximal torus and a Borel subgroup $T\subset B\subset G$ are picked, we write $\Lambda$ for the coweights lattice of $T$, and $\Lambda^+$ for the dominant coweights of $G$. 
   
\sssec{} 
\label{Sect_1.1.2_conventions}
Recall the categories $\DGCat_{cont}, \DGCat^{non-cocmpl}$ defined in (\cite{GR}, ch. I.1, Section 10.3.1 and 10.3.3). For $C\in\DGCat_{cont}$ equipped with a t-structure, $C^{\heartsuit}$ denotes the heart of $C$. For a prestack $Y$ locally of finite type we write $Shv(Y)^{constr}\subset Shv(Y)$ for the $\DG$-subcategory of constructible objects defined in (\cite{AGKRRV}, Section~F.2). 

 Similarly, let $Y$ be an ind-algebraic stack $Y$ written as $Y\,\iso\,\colim_{i\in I} Y_i$  with $I$ is a filtered $\infty$-category, $Y_i$ an algebraic stack locally of finite type such that for a map $i\to i'$ in $I$ the transition map $\alpha_{i,i'}: Y_i\to Y_{i'}$ is a closed immersion. By definition, $Shv(Y)\,\iso\, \lim_{i\in I^{op}} Shv(Y_i)$ in $\DGCat_{cont}$ with respect to the $!$-restrictions. Passing to the left adjoints, we can also rewrite it as $Shv(Y)\,\iso\,\colim_{i\in I} Shv(Y_i)$ in $\DGCat_{cont}$ with respect to the $!$-direct images. For $i\to i'$ as above $(\alpha_{i,i'})_!: Shv(Y_i)^{constr}\to Shv(Y_{i'})^{constr}$, and we get a diagram 
$$
I\to \DGCat^{non-cocmpl}, \; i\mapsto Shv(Y_i)^{constr}
$$
Define $Shv(Y)^{constr}$ as $\colim_{i\in I} Shv(Y_i)^{constr}$ taken in $\DGCat^{non-cocmpl}$. Note that $Shv(Y)^{constr}\subset Shv(Y)$ is a full subcategory.\footnote{This follows from \cite{Ro}.}  

 Let $1-\Cat$ is the $\infty$-category of $\infty$-categories introduced in (\cite{GR}, ch. I.1, Section~1.1.1). Recall that the forgetful functor $\DGCat^{non-cocmpl}\to 1-\Cat$ preserves filtered colimits. So, $Shv(Y)^{constr}\subset Shv(Y)$ is the full subcategory of those objects, which come as $!$-direct image under $Y_i\to Y$ of some object of $Shv(Y_i)^{constr}$. 
 
 For the convenience of the reader recall the definition of the perverse t-structure on $Y$. By definition, $Shv(Y)^{\le 0}$ is the smallest full subcategory containing $Shv(Y_i)^{\le 0}$ for $i\in I$, closed under extensions and small colimits. This t-structure is accessible and compatible with filtered colimits. The inclusion $Shv(Y)^{constr}\subset Shv(Y)$ is compatible with this t-structure, so $Shv(Y)^{constr}$ inherits a t-structure from $Shv(Y)$.

\sssec{Acknowledgements} I am grateful to Sam Raskin for answering my questions and very useful discussions. I also thank Joakim Faergeman for his comments about a preliminary version of this paper and fruitful email correspondence.

\section{Main results and conjectures} 
\label{Sect_1.2}  

\ssec{Fourier coefficients} 
\sssec{} Let $x\in\gg$ be nilpotent and $\OO:=\OO_x\subset\gg$ the nilpotent orbit through $x$. Let $\sigma: \gsl_2\to \gg$ be the corresonding $\SL_2$-triple, so $x$ is the nilpositive element of this triple (in the sense of \cite{Ra}, Definition~4.2.1). Write $(H, x, Y)$ for the correspodning standard $\gsl_2$-triple in $\gg$.

Let $\gg_i\subset \gg$ be the subspace, on which $\sigma(\Gm)$ acts by $i$, so $\gg=\oplus_i \gg_i$. Let $P\subset G$ be the parabolic such that $\gp:=\Lie P=\oplus_{i\ge 0}\, \gg_i$ is the Jacobson-Morozov parabolic subalgebra of $x$ (\cite{Ra}, Definition~4.3.2). Let $U\subset P$ be its unipotent radical, then $\gu:=\Lie U=\oplus_{i>0} \gg_i$. We also get the subgroup $V\subset U$ such that $\gv:=\Lie V=\oplus_{i\ge 2}\, \gg_i$. Let $M\subset P$ be the Levi such that $\gm=\Lie M=\gg_0$. Set $\gv'=\oplus_{i\ge 3}\; \gg_i$. 
By the theory of $\gsl_2$-representations, $\gv'\subset [\gu, x]$. Let $V'\subset V$ be the subgroup such that $\Lie V'=\gv'$. So, $\gv/\gv'\,\iso\, \gg_2$ as $M$-modules.

 Recall that $x$ is called even iff $\gg_1=0$ (\cite{Ra}, Definition~4.3.6). If $x$ is even then the above shows that $\gg_2\,\iso\, \gv/[\gv,\gv]$.  
 
 By (\cite{CM}, Lemma~4.1.4), $Px=\OO\cap \gv$ is open and dense in $\gv$. More precisely, let 
$$
\cP=\{Z\in\gg_2\mid \gg^Z\cap \gg_{-2}=0\}
$$ 
Then $\cP\subset \gg_2$ is the open $M$-orbit on $\gg_2$, and $\OO\cap \gv=\cP+\gv'=Px$. 
 
 Similarly, let $\cP^-=\{Z\in\gg_{-2}\mid \gg^Z\cap \gg_2=0\}$. Choosing an opposite Borel in $\gsl_2$, from the above one gets the following. The subscheme $\cP^-\subset \gg_{-2}$ is the open $M\times\Gm$-orbit on $\gg_{-2}$. Note that $x\in\cP$ and $Y\in \cP^-$, 
 

 Let $\gm^x$ be the stabilizer of $x$ in $\gm$. Then $\gm^x$ is reductive (cf. \cite{Ra}, Lemma~4.3.4). Note that for any $i\ge 0$, the restriction of $B$ provides a nondegenerate pairing $\gg_i\times\gg_{-i}\to k$. 
 
  We underline that if $x\ne 0$ then $\gg_2\ne 0$. 
  
\sssec{}  Let $(V/V')^*_0\subset (V/V')^*$ be the open $M\times\Gm$-orbit. It is understood that $\Gm$ acts by scalar multiplications. If $p=0$ then we identify $(V/V')^*$ with $\gg_{-2}$ as above, so this open orbit is $\cP^-$.   
 
\sssec{} In the case of $x$ odd we have $[\gv,\gv]=\gp_{\ge 4}:=\oplus_{i\ge 4} \,\gg_i$. Indeed, for $m\ge 2$ the map $ad_x: \gg_{m-2}\to \gg_m$ is surjective by the representation theory of $\gsl_2$. So, in this case 
$$
\gv/[\gv,\gv]\,\iso\, \gg_2\oplus\gg_3
$$ 
 
\sssec{} Consider the diagram 
$$
\Bun_{P/V'}\getsup{\gp} \Bun_P\toup{\gq} \Bun_G
$$ 
For $K\in Shv(\Bun_G)$ we want to consider the Fourier transform of $\gp_!\gq^*K$ with respect to $V/V'$.


\sssec{} If $\cF$ is a $P/V'$-torsor on $X$ write $\cF_M$ for the induced $M$-torsor on $X$, and similarly for $\cF_{P/V}$. 
Since $V/V'$ is a linear representation of $M$, $(V/V')_{\cF_M}$ is a vector bundle on $X$. A datum of a torsor $\cF'$ under $(V/V')_{\cF_M}$ on $X$ is the same as an exact sequence of vector bundles on $X$
\begin{equation}
\label{seq_one}
0\to (V/V')_{\cF_M}\to ?\to \cO_X\to 0
\end{equation}
Namely, $\cF'$ is the torsor of sections of (\ref{seq_one}) over $1$. 


The stack $\Bun_{P/V'}$ classifies $\cF_M\in\Bun_M$ and a torsor under $(U/V')_{\cF_M}$ on $X$. 

\sssec{Case of $x$ even} If $x$ is an even nilpotent element (\cite{Ra}, Definition~4.3.6) then $V=U$. In this case $\Bun_{P/V'}$ classifies $\cF_M\in\Bun_M$ and an exact sequence (\ref{seq_one}) on $X$. 
 
Let $\cY_M$ be the stack classifying $\cF_M\in\Bun_M$, and a section on $X$
\begin{equation}
\label{map_s_first}
s: (V/V')_{\cF_M}\to\Omega
\end{equation}
In this case the Fourier transform functor 
$$
\Four: Shv(\Bun_{P/V'})\to Shv(\cY_M)
$$    
is defined as follows. We have a diagram of projections
$$
\begin{array}{ccc}
&\A^1\\
&\uparrow\lefteqn{\scriptstyle \ev}\\
\cY_M \getsup{a} & \cY_M\times_{\Bun_M} \Bun_{P/V'}
 & \toup{b} \Bun_{P/V'},
\end{array}
$$
and the map $\ev$ sends a point of $\cY_M\times_{\Bun_M} \Bun_{P/V'}$ to the pairing of (\ref{map_s_first}) with (\ref{seq_one}). 

If $p>0$ then for $\cK\in Shv(\Bun_{P/V'})$ set
$$
\Four(\cK)=(a)_!b^*\cK\otimes\ev^*\cL_{\psi}[\dimrel(b)]
$$
Though $\Four$ depends on $\psi$, we do not express this in our notation. 

 Let $\cY_M^0\subset \cY_M$ be the open substack given by the property that $s$ at the generic point of $X$ lies in $(V/V')^*_0$.
 
\sssec{Case of $x$ odd} 
\label{Sect_1.2.6}
Assume $x$ odd. Let $\cX_P$ be the stack classifying  $\cF\in\Bun_{P/V'}$ and a $(V/V')_{\cF_M}$-torsor $\bar\cF$ on $X$ given also by an exact sequence (\ref{seq_one}). Then $(V/V')_{\cF_M}$ is a subsheaf of the sheaf of automorphisms of the $P/V'$-torsor $\cF$. We get the action map $\act: \cX_P\to \Bun_{P/V'}$ sending $(\cF,\bar\cF)$ to the quotient of $\cF\times\bar\cF$ by the diagonal action of $(V/V')_{\cF_M}$. 

 In fact, $\cX_P$ identifies with the stack classifying $\cF,\cF'\in\Bun_{P/V'}$ and an isomorphism 
$$
\mu_P:\cF\times_{P/V'} P/V\,\iso\, \cF'\times_{P/V'} P/V
$$ 
of $P/V$-torsors on $X$. Namely, for $(\cF,\cF',\mu_P)$ as above let $\bar\cF$ be the sheaf of isomorphisms $\cF\,\iso\, \cF'$ of $P/V'$-torsors on $X$ compatible with $\mu_P$. Then $(V/V')_{\cF_M}$ acts on $\bar\cF$ via its action on $\cF$, and in this way $\bar\cF$ becomes a $(V/V')_{\cF_M}$-torsor on $X$. 

 Let $\cY_P$ be the stack classifying $\cF\in \Bun_{P/V'}$ and a section (\ref{map_s_first}). We define the Fourier transform 
$$
\Four_P: Shv(\Bun_{P/V'})\to Shv(\cY_P)
$$ 
as follows. Consider the diagram 
$$
\begin{array}{ccc}
& \AA^1\\
& \uparrow\lefteqn{\scriptstyle \ev_P}\\
\cY_P\getsup{a_P} &\cY_P\times_{\Bun_{P/V'}} \cX_P & \toup{b_P} \cX_P\toup{\act} \Bun_{P/V'},
\end{array}
$$
where $a_P, b_P$ are the projections. The map $\ev_P$ sends a point of $\cY_P\times_{\Bun_{P/V'}} \cX_P$ to the pairing of (\ref{seq_one}) with (\ref{map_s_first}). In the case $p>0$ for $\cK\in Shv(\Bun_{P/V'})$ set 
$$
\Four_P(\cK)=(a_P)_! b_P^*\act^*\cK\otimes\ev_P^*\cL_{\psi}[\dimrel(\act\comp b_P)]
$$
  
Let $\cY_P^0\subset \cY_P$ be the open substack given by the property that $s$ at the generic point of $X$ lies in $(V/V')^*_0$.
 
\begin{Def} We say that $K\in Shv(\Bun_G)$ has no Fourier coefficients corresponding to $\OO$ if 
\begin{itemize}
\item $\Four_P(\gp_!\gq^*K)$ vanishes over $\cY_P^0$ for $x$ is odd;
\item $\Four(\gp_!\gq^*K)$ vanishes over $\cY_M^0$ for $x$ even.
\end{itemize}

 Let $\cF_{\OO}\subset Shv(\Bun_G)$ be the full subcategory of those $K\in Shv(\Bun_G)$ which have no Fourier coefficients corresponding to $\OO$.
\end{Def}

 Note that $\cF_{\OO}\in\DGCat_{cont}$, the embedding $\cF_{\OO}\hook{} Shv(\Bun_G)$ is continuous. 
 
\sssec{} Write $\Nilp$ for the global nilpotent cone as defined in (\cite{AGKRRV}, Appendix D). Recall that it can be seen as a mapping stack
$$
\Maps(X, \cN(\gg)/(G\times\Gm))\times_{\Maps(X, B(\Gm))}\{\Omega_X\}
$$
Write $Shv_{\Nilp}(\Bun_G)\subset Shv(\Bun_G)$ for the $\DG$-subcategory of sheaves with singular support in $\Nilp$ as defined in 
(\cite{AGKRRV}, Section F.8). 

More generally, for a closed $G$-invariant subscheme $Y\subset \cN(\gg)$ write $Shv_Y(\Bun_G)\subset Shv_{\Nilp}(\Bun_G)$ for the full $\DG$-subcategory of those objects, whose singular support is contained in 
$$
\Maps(X, Y/(G\times\Gm))\times_{\Maps(X, B(\Gm))}\{\Omega_X\}
$$
 
\begin{Def} 
\label{Def_2.1.9}
Let $F_{\ov{\OO}}\subset Shv(\Bun_G)$ be the full subcategory equal to $\underset{\OO'\nsubseteq \ov{\OO}}{\cap}\cF_{\OO'}$. 
More generally, for a closed $G$-invariant subscheme $Y\subset \cN(\gg)$ let $F_Y\subset Shv(\Bun_G)$ be the full subcategory $\underset{\OO'\nsubseteq Y}{\cap}\cF_{\OO'}$. Let $F_{\Nilp, Y}$ be the intersection $F_Y\cap Shv_{\Nilp}(\Bun_G)$ inside $Shv(\Bun_G)$.
\end{Def} 

 Note that $F_Y\in\DGCat_{cont}$, the embedding $F_Y\hook{} Shv(\Bun_G)$ is continuous. If $\OO\subset \ov{\OO}'$ are nilpotent orbits then $F_{\ov{\OO}}\subset F_{\ov{\OO}'}$ is a full continuous embedding (and the same for $\Nilp$-versions).
  
\begin{Thm} 
\label{Th_preservation_Hecke}
For any nilpotent orbit $\OO$ in $\gg$, $\cF_{\OO}$ is preserved by Hecke functors acting on $Shv(\Bun_G)$.
\end{Thm} 
  
\begin{Rem} Theorem~\ref{Th_preservation_Hecke} implies that for $\OO\in\cX(\gg)$, the full $\DG$-subcategory $F_{\Nilp, \ov{\OO}}$ of $Shv_{\Nilp}(\Bun_G)$ is stable under the action of Hecke functors.
\end{Rem}

\begin{Con} 
\label{Con_2.1.13_equality}
Let $Y\subset \cN(\gg)$ be a closed $G$-invariant subscheme. 
Then 
$$
F_{\Nilp, Y}= Shv_Y(\Bun_G)
$$ 
\end{Con} 

\begin{Rem} i) For $Y=\cN(\gg)-\OO_{reg}$ and the sheaf theory being ind-holonomic $\cD$-modules
 Conjecture~\ref{Con_2.1.13_equality} is proved in the striking paper (\cite{FR}, Theorems B, C). They show that $Shv_Y(\Bun_G)$ coincides with the category of the so-called anti-tempered objects of $Shv_{\Nilp}(\Bun_G)$. Their argument does not seem to work for $\ell$-adic sheaves for example. 
 
\smallskip\noindent 
ii) Assume that the sheaf theory is ind-holonomic $\cD$-modules, and $G$ is semi-simple and simply-connected. Then for $Y=\{0\}$ one has the inclusion $Shv_Y(\Bun_G)\subset F_{\Nilp, Y}$. Indeed, by (\cite{Fa}, Proposition 4.1.4.1), any irreducible local system on $\Bun_G$ is trivial, the claim easily follows. 
\end{Rem}

\sssec{} If we need to underline the dependence on $G$, we will write $\cF_{\OO}(G)$, $F_{\bar\OO}(G)$ and so on. 

\sssec{} The Lusztig-Spaltenstein duality map $d: \cX(\gg)\to \cX(\check{\gg})$ is constructed in (\cite{Sp}, \cite{BV}). This is the order reversing map denoted by $d_{BV}$ in \cite{Ach}. The image of $d$ is called the set of special nilpotent orbits $\cX(\check{\gg})^{sp}\subset \cX(\check{\gg})$. We similarly have $d: \cX(\check{\gg})\to \cX(\gg)$, and the restriction $\cX(\gg)^{sp}\to \cX(\check{\gg})^{sp}$ of $d$ is known to be bijective. One more important property of $d$ is that for any $\OO\in \cX(\gg)$, $\OO\subset\ov{d^2(\OO)}$. By (\cite{BV}, Corollary~A.3), for $\OO\in \cX(\gg)$, $d^2(\cO)$ is the unique smallest special nilpotent orbit dominating $\OO$. 

 In general, $d$ exchanges the zero orbit $\OO_0$ and the regular nilpotent orbit $\OO_{reg}$.  
 
 If $\gg$ is simple it admits the subregular nilpotent orbit $\OO_{subreg}$, the unique open orbit in $\Nilp-\OO_{reg}$. It is known that $d$ exchanges $\OO_{subreg}$ and the so called \select{minimal special} nilpotent orbit. (In general, the minimal nonzero nilpotent orbit $\OO_{min}$ is not special).
 
\sssec{} If $\OO\in\cX(\gg)$, let ${}^<\OO^{sp}\subset \cN(\gg)$ be the union of $\bar\OO'$ for all $\OO'\in \cX(\gg)^{sp}$, which are strictly less than $\OO$. Definition~\ref{Def_2.1.9} for
 $Y={}^<\OO^{sp}$ gives a full subcategory  $F_{{}^<\OO^{sp}}\subset F_{\bar\OO}$. For $\OO\ne \OO_0$ set 
$$
F_{\Nilp, {}^<\OO^{sp}}=F_{{}^<\OO^{sp}}\cap Shv_{\Nilp}(\Bun_G)
$$
If $\OO=\OO_0$ we let $F_{\Nilp, {}^<\OO^{sp}}$ denote the $\DG$-subcategory of $Shv_{\Nilp}(\Bun_G)$ containing only the zero object. 

 The first question, independent of the Langlands correspondence is to see which successive subquotients in the filtration $F_{\bar\OO}$ for $\OO\in \cX(\gg)$ vanish. We expect that restricting the filtration to $\cX(\gg)^{sp}$, we do not lose any information. More precisely, we propose the following.
 
\begin{Con} 
\label{Con_2.1.15}
Let $\OO\in \cX(\gg)$ be non special. Then $F_{\bar\OO}=F_{{}^<\OO^{sp}}$.
\end{Con}

\sssec{Example} As a motivation for Conjecture~\ref{Con_2.1.15}, consider the first nontrivial example of $G=\Sp_4$. The Hasse diagram of nilpotent orbits in $\gg$ is
$$
\bar\OO_0 \subset \bar \OO_{min}\subset \bar\OO_{subreg}\subset \bar\OO_{reg}
$$
in this case. In the notations of (\cite{CM}, Section~5.1), they are given by the following partitions $\OO_0=(1^4), \OO_{min}=(2, 1^2), \OO_{subreg}=(2^2), \OO_{reg}=(4)$. 
All of them are special except $\OO_{min}$. So, in this case we expect $F_{\bar\OO_{min}}=F_{\bar\OO_0}$. 

 This equality is motivated by a result in the classical theory of automorphic forms (\cite{Novo}, Theorem~3, see also \cite{H1}) saying that any infinite-dimensional representation of $\Sp_4$ over a local non-archimedian field admits either a Whittaker or Bessel model. The analog of this in the global setting would be that an infinite-dimensional automorphic representation of $G(\AA)$ admits a nonzero Fourier coefficient either for $\OO_{reg}$ or $\OO_{subreg}$, here $\AA$ is the ring of ad\`eles of a curve over a finite field. 
 
\sssec{} A much more general support for Conjecture~\ref{Con_2.1.15} is provided by the main result of \cite{JLS}, which is concerned precisely with an analog of Conjecture~\ref{Con_2.1.15} at the classical level of automorphic forms (they also study an analogous question for representations of reductive groups over local non-archimedian fields). 

\sssec{Question} Consider the case of $X=\PP^1$. Then one has the Shatz stratification of $\Bun_G$ by locally closed substacks $\Bun_{G,\lambda}$ indexed by $\lambda\in\Lambda^+$. Write $\IC_{\lambda}$ for the $\IC$-sheaf of the stratum $\Bun_{G,\lambda}$. Given $\lambda\in\Lambda^+$ find the smallest closed $G$-invariant subset $Y\subset \cN(\gg)$ such that $\IC_{\lambda}\in F_Y$. 
 
\ssec{Spectral part and the Langlands correspondence} 

\sssec{} Recall the algebraic stack $\Arth_{\check{G}}(X)$ over $\Spec e$ defined in (\cite{AGKRRV}, 14.2.2). In the notations of \select{loc.cit.} it classifies right t-exact symmetric monoidal functors 
$$
h: \Rep(\check{G})\to \QLisse(X),
$$ 
which we think of as a $\check{G}$-local system $\sigma$ on $X$, and $A\in\H^0(X, h(\check{\gg}^*))$. As in \select{loc.cit.}, let $\Nilp\subset \Arth_{\check{G}}(X)$ denote the closed conical subset of those $(\sigma, A)$ that $A$ is nilpotent. That is, for any local on $X$ trivialization of $\sigma$, one requires that $A$ takes values in $\cN(\check{\gg})\subset \check{\gg}^*$. Here we used the isomorphism $\check{\kappa}$ fixed in Section~\ref{Sect_1.1.2_on_conventions}. Given $\OO'\in \cX(\check{\gg})$, one similarly gets the locus denoted $\ov{\OO}'\subset \Arth_{\check{G}}(X)$ by abuse of notations.

This gives a filtration on 
\begin{equation}
\label{category_for_1.2.12}
\IndCoh_{Nilp}(\LocSys_{\check{G}}^{restr}(X))
\end{equation}
indexed by nilpotent orbits in $\check{\gg}$. Namely, for $\OO'\in \cX(\check{\gg})$ one has the full $\DG$-subcategory
$$
S_{\ov{\OO}'}=\IndCoh_{\ov{\OO}'}(\LocSys_{\check{G}}^{restr}(X)) 
$$
If $\OO\subset \ov{\OO}'$ then $S_{\ov{\OO}}\subset S_{\ov{\OO}'}$. 

 The following is our main conjecture, it proposes a geometric counterpart of (\cite{J1}, Conjectures~4.2 and 4.3). 

\begin{Con} 
\label{Con_1.2.14}
For any $\OO\in \cX(\gg)^{sp}$ there is an equivalence
\begin{equation}
\label{conj_LL_for_any_OO}
Shv_{\Nilp}(\Bun_G)/ F_{\Nilp, {}^<\OO^{sp}}\,\iso\, \IndCoh_{\ov{d(\OO)}}(\LocSys_{\check{G}}^{restr}(X))
\end{equation}
compatible with the Hecke actions. Here $d: \cX(\gg)\to \cX(\check{\gg})$ is the Lusztig-Spaltenstein duality map. For $\OO=\OO_0$ this recovers the conjectural geometric Langlands equivalence 
\begin{equation}
\label{global_Langlands_equiv_constr_context}
Shv_{\Nilp}(\Bun_G)\,\iso\,\IndCoh_{Nilp}(\LocSys_{\check{G}}^{restr}(X))
\end{equation}
of (\cite{AGKRRV}, Conjecture~21.2.7). 
\end{Con}

\begin{Rem} Take $\OO=\OO_{reg}\in \cX(\gg)^{sp}$ and $Y=\cN(\gg)-\OO_{reg}$. In view of Conjectures~\ref{Con_2.1.15} and \ref{Con_2.1.13_equality} we expect that 
$$
Shv_{\Nilp, {}^<\OO^{sp}}=F_{\Nilp, Y}=Shv_Y(\Bun_G)
$$ 
coincides with the category $Shv_{\Nilp}(\Bun_G)^{anti-temp}$ of anti-tempered objects of 
$$
Shv_{\Nilp}(\Bun_G)
$$ 
In this case (\ref{conj_LL_for_any_OO}) would become the tempered geometric Langlands conjecture as formulated in (\cite{FR}, Section~1.6.2):
$$
Shv_{\Nilp}(\Bun_G)/Shv_{\Nilp}(\Bun_G)^{anti-temp}\,\iso\, \QCoh(\LocSys_{\check{G}}^{restr}(X))
$$
\end{Rem}
\begin{Rem} Assume $G$ simple and take $\OO$ to be the minimal special nilpotent orbit $\OO^{sp}_{min}$ in $\cN(\gg)$, so $d(\OO)=\OO_{subreg}$. By definition, ${}^<\OO^{sp}=\{0\}$. In this case assuming Conjecture~\ref{Con_2.1.13_equality} we get 
 $F_{\Nilp, {}^<\OO^{sp}}=Shv_0(\Bun_G)$. So, Conjecture~\ref{Con_1.2.14} predicts in this case an equivalence
$$
Shv_{\Nilp}(\Bun_G)/Shv_0(\Bun_G)\,\iso\, \IndCoh_{\ov{\OO}_{subreg}}(\LocSys_{\check{G}}^{restr}(X))
$$
If our sheaf theory is that of ind-holonomic $\cD$-modules then the latter equivalence is established in \cite{Fa}.
\end{Rem}

\ssec{Compatibility with the parabolic induction} 

\sssec{} The idea of the compatibility of Conjecture~\ref{Con_1.2.14} with the parabolic induction essentially appears already in  David Ginzburg's paper \cite{Ginz}. 

 Recall the definition of induced orbits (\cite{CM}, Section~7.1). For a parabolic subalgebra $\gp\subset \gg$ with Levi decomposition $\gp=\gm\oplus\gu$ let $\OO\in\cX(\gm)$. By (\cite{CM}, Theorem~7.1.1), there is a unique nilpotent orbit 
$\Ind_{\gp}^{\gg}(\OO)\in \cX(\gg)$ such that 
$$
\Ind_{\gp}^{\gg}(\OO)\cap (\OO+\gu)\subset \OO+\gu
$$ 
is an open dense subset. It is \select{the nilpotent orbit in $\gg$ induced from $\OO$}. Recall that $\dim \Ind_{\gp}^{\gg}(\OO)=\dim\OO+2\dim\gu$, 
and $\Ind_{\gp}^{\gg}(\OO)$ depends only on $\gm$ and not on the choice of a parabolic subalgebra containing $\gm$ (\cite{CM}, Theorem~7.1.3). Therefore, we also write $\Ind_{\gm}^{\gg}(\OO):=\Ind_{\gp}^{\gg}(\OO)$. The induction of the zero orbit $\Ind_{\gp}^{\gg}(\OO_0)$ is called \select{a Richardson orbit}. 

 The obtained map $\Ind_{\gm}^{\gg}: \cX(\gm)\to\cX(\gg)$ has the following remarkable property. 

\begin{Pp}[(\cite{CM}, Theorem~8.3.1), (\cite{BV}, Proposition~A.2)] Let $\OO_{\gm}\in \cX(\gm)$ and $\cO\in \cX(\gg)$ containing $\OO_{\gm}$. Then one has
$$
d(\OO))=\Ind_{\check{\gm}}^{\check{\gg}}(d(\OO_{\gm}))
$$
\end{Pp}

\sssec{} For a parabolic subgroup $P\subset G$ with Levi quotient $M$ consider the diagram of natural maps
\begin{equation}
\label{diag_for_Sect_2.2.3}
\Bun_M\getsup{\gq^P} \Bun_P\toup{\gp^P}\Bun_G
\end{equation}
The Eisenstein series functor $\Eis_!: Shv(\Bun_M)\to Shv(\Bun_G)$ is defined by
$$
\Eis_!(\cK)=\gp^P_!(\gq^P)^*\cK
$$
\sssec{} Let also $\Bunt_P$ be the Drinfeld compactification of $\Bun_P$ defined in \cite{BG}. More precisely, if $[G,G]$ is simply-connected then $\Bun_P$ is defined as in \select{loc.cit.}, and in general one corrects this definition as in (\cite{Sch}, Section~7.4).
Then $\Bun_P\subset \Bunt_P$ is an open substack, and the diagram (\ref{diag_for_Sect_2.2.3}) extends to the following one
$$
\Bun_M\,\getsup{\tilde\gq^P} \,\Bunt_P\,\toup{\tilde\gp^P}\,\Bun_G
$$
The compactified Eisenstein series functor $\Eis_{!*}: Shv(\Bun_M)\to Shv(\Bun_G)$
is defined as in \cite{BG} by
$$
\Eis_{!*}(\cK)=\tilde\gp^P_!((\tilde\gq^P)^*\cK\otimes \IC_{\Bunt_P}),
$$
here $\IC_{\Bunt_P}$ is the $\IC$-sheaf on $\Bunt_P$.

\sssec{} It is expected that both $\Eis_!$ and $\Eis_{!*}$ restrict to the functors $Shv_{\Nilp}(\Bun_M)\to Shv_{\Nilp}(\Bun_G)$, that is, preserve the nilpotence of the singular support. 

 Moreover, the so obtained conjectural functor 
$$
\Eis_!: Shv_{\Nilp}(\Bun_M)\to Shv_{\Nilp}(\Bun_G)
$$ 
is expected to be compatible with the conjectural equivalence (\ref{global_Langlands_equiv_constr_context}) in the same sense as in the context of $\cD$-modules. Recall that the analog of $\Eis_!$ for the $\cD$-module context is compatible with the conjectural equivalence (\ref{conj_global_Langlands_for_D-mod}) in the sense of (\cite{AG}, Section~1.2.3). 

Based on the existing works at the level of functions (\cite{GRS2}-\cite{JL2}), we suggest the following.
\begin{Con} 
\label{Con_2.2.6}
Let $\OO\in\cX(\gm)$, $\OO'=\Ind_{\gm}^{\gg}(\OO)\in\cX(\gg)$
 and $K\in F_{\bar\OO}(M)$. Then $\Eis_!(K)\in F_{\bar\OO'}(G)$. 
\end{Con}  

\sssec{Example} The constant sheaf $e_{\Bun_M}$ on $\Bun_M$ lies in $F_{\bar\OO_0}(M)$. Let $\OO'=\Ind_{\gm}^{\gg}(\OO_0)$ be the corresponding Richardson orbit.
For $G=\GL_n$ the analog at the level of functions of the property $\Eis_!(e_{\Bun_M})\in F_{\bar\OO'}(G)$ appears as (\cite{Ginz}, Conjecture~5.1). 

\sssec{Question} Should we expect the analog of Conjecture~\ref{Con_2.2.6} with $\Eis_!$ replaced by $\Eis_{!*}$?

\section{Examples}
\label{Sect_Examples}

In this section we prove Theorem~\ref{Th_preservation_Hecke} for a regular nilpotent $x$. We also discuss the example of the minimal nilpotent orbit. 

\ssec{Regular nilpotent orbit}
\label{Sect_2.1}

\sssec{} Pick a maximal torus and a Borel subgroup $T\subset B\subset G$. Let $\cI$ denote the set of vertices of the Dynkin diagram. For $i\in \cI$ let $\alpha_i$ (resp., $\check{\alpha}_i$) denote the corresponding coroot (resp., root) of $(T,G)$. For $i\in\cI$ pick a nonzero vector $x_{\check{\alpha}_i}$ in the coresponding root subspace $X_{\check{\alpha}_i}$ of $\gg$. Take $x=\sum_{i\in\cI} x_{\check{\alpha}_i}$. Recall that $x$ is always a distinguished nilpotent (\cite{Ra}, Definition~4.3.1), and the corresponding Jacobson-Morozov parabolic subalgebra is $\gb=\Lie B$.  Recall that each distinguished nilpotent is even by (\cite{CM}, Theorem~8.2.3), so $x$ is even. Let $\OO$ denote the nilpotent orbit of $\gg$ containing $x$, we refer to it as the regular nilpotent orbit. 
By $w_0$ we denote the longest element of the Weyl group of $(T, G)$.

In the rest of Section~\ref{Sect_2.1} we prove Theorem~\ref{Th_preservation_Hecke} for $x$ regular.

\sssec{} Let $\check{\Lambda}$ be the character lattice of $T$, $\Lambda$ be the cocharacter lattice. For a $T$-torsor on some base and $\check{\lambda}\in\check{\Lambda}$ let $\cL^{\check{\lambda}}_{\cF_T}$ denote the line bundle obtained from $\cF_T$ via extension of scalars by $\check{\lambda}: T\to\Gm$. 

\sssec{} We have $V/V'\,\iso\, \oplus_{i\in\cI}  X_{\check{\alpha}_i}$ as representations of $T$. The stack $\Bun_{B/V'}$ classifies pairs $\cF_T\in\Bun_T$, and an exact sequence on $X$
\begin{equation}
\label{seq_for_2.1.3}
0\to \underset{i\in\cI}{\oplus} \cL^{\check{\alpha}_i}_{\cF_T}\to ?\to \cO_X\to 0
\end{equation}
The stack $\cY_T$ classifies $\cF_T\in\Bun_T$ and a section 
$$
s: \underset{i\in\cI}{\oplus} \cL^{\check{\alpha}_i}_{\cF_T}\to\Omega
$$ 
on $X$. Write $s_{\check{\alpha}_i}: \cL^{\check{\alpha}_i}_{\cF_T}\to\Omega$ for the corresponding component of $s$. The open substack $\cY^0_T\subset \cY_T$ is given by the property that for any $i\in\cI$, $s_{\check{\alpha}_i}\ne 0$.
 
\sssec{} Pick a closed point $y\in X$. The Hecke functors 
$$
\H^{\la}_y, \H^{\ra}_y: \Rep(\check{G})\times Shv(\Bun_G)\to Shv(\Bun_G)
$$ 
are defined as in (\cite{BG}, 3.2.4). 
This definition is also equivalent to the one from (\cite{FGV}, 5.3.6). Let $\Lambda^+$ be the set of dominant coweights. Let $\Gr_{G,y}$ denote the affine grassmanian of $G$ at $y$. For $\lambda\in\Lambda^+$ we denote by $\cA^{\lambda}_G$ the $\IC$-sheaf of $\ov{\Gr}_G^{\lambda}$ as in (\cite{BG}, 3.2.1). 

 Let $_y\cH_G$ be the Hecke stack classifying $\cF_G,\cF'_G\in\Bun_G$ and $\beta: \cF_G\,\iso\, \cF'_G\mid_{X-y}$. We have the diagram
$$
\Bun_G\getsup{h^{\la}} {_y\cH_G}\toup{h^{\ra}} \Bun_G,
$$
where $h^{\la}$ (resp., $h^{\ra}$) sends the above point to $\cF_G$ (resp., to $\cF'_G$). 
Let $\cG_y\to\Bun_G$ be the $G(\cO_y)$-torsor classifying $\cF_G\in\Bun_G$ and its trivialization over the formal disk $D_y$ around $y$. We have the isomorphisms 
$$
\id^l, \id^r: {_y\cH_G}\,\iso\, \cG_y\times^{G(\cO_y)} \Gr_{G, y}
$$
such that the projection to $\Bun_G$
corresponds to $h^{\la}, h^{\ra}:{_y\cH_G}\to\Bun_G$ respectively. For $\cT\in Shv(\Bun_G), \cS\in \Perv_{G(\cO_y)}(\Gr_{G,y})$ we have the corresponding twisted products $(\cT\tboxtimes\cS)^l, (\cT\tboxtimes\cS)^r$ normalized as in \select{loc.cit}. We then set
$$
\H^{\la}_y(\cS, \cT)=
\cS\ast \cT=h^{\ra}_!(\cT\tboxtimes (\ast\cS))^l
$$
and 
$$
\H^{\ra}_y(\cS, \cT)=
\cT\ast \cS=h^{\la}_!(\cT\tboxtimes \cS)^r
$$

 Here we denoted by $\ast$ the covariant self-functor on $\Perv_{G(\cO_y)}(\Gr_{G,y})$ coming from $G(\cO_y)\to G(\cO_y), g\mapsto g^{-1}$. 
 
 For $\lambda\in\Lambda^+$ let $_y\ov{\cH}_G^{\lambda}$ be the closed substack of $_y\cH_G$ which identifies under $\id^l$ with $\cG_y\times^{G(\cO_y)}\ov{\Gr}_G^{
\lambda}$. This notation agrees with that of (\cite{BG}, Section~3.2.4). For a point $(\cF_G,\cF'_G,\beta)\in {_y\ov{\cH}_G^{\lambda}}$ we say that $\cF'_G$ is in the position $\le\lambda$ with respect to $\cF_G$. This is equivalent to $\cF_G$ being in the position $\le -w_0(\lambda)$ with respect to $\cF'_G$. 
 
\sssec{} Let $K\in \cF_{\OO}$ and $\lambda\in\Lambda^+$. We will show that $\cA^{\lambda}_G\ast K\in \cF_{\OO}$. 

\sssec{} Pick a $k$-point of $\cY^0_T$ given by $(\cF'_T, s')$. We check that the $*$-fibre of 
$$
\Four(\gp_!\gq^*(\cA^{\lambda}_G\ast K))
$$ 
at this point vanishes. Our argument will be compatible with the field extensions $k\to k'$, so this is sufficient.

Let $\Lambda_{G_{ad}}^+$ be the set of dominant coweights for $G_{ad}=G/Z(G)$. Let 
$$
cond(s')=\sum_{z\in X} cond(s')_z
$$ 
be the $\Lambda_{G_{ad}}^+$-valued divisor on $X$ given by the property: if $i\in\cI$ then $\<cond(s')_z,\check{\alpha}_i\>$  is the order of zero of $s'_{\check{\alpha}_i}: \cL^{\check{\alpha}_i}_{\cF_T}\hook{}\Omega$ at $z$.

\sssec{} 
Let $\Bun_U^{\cF'_T}$ be the stack classifying $\cF'_B\in\Bun_B$ and an isomorphism $\gamma': {\cF'_B\times_B T}\iso \cF'_T$. Denote by $\ev_{s'}: \Bun_U^{\cF'_T}\to\Ga$ the composition
$$
\Bun_U^{\cF'_T}\to \prod_{i\in\cI} \H^1(X, \cL^{\check{\alpha}_i}_{\cF'_T})\toup{s'}\prod_{i\in\cI} \H^1(X, \Omega)\,\iso\, \prod_{i\in\cI} \Ga\toup{\sum}\Ga
$$
Let $q': \Bun_U^{\cF'_T}\to\Bun_G$ be the natural map. It suffices to show that 
\begin{equation}
\label{complex_for_2.1.7}
\RG_c(\Bun_U^{\cF'_T}, \ev_{s'}^*\cL_{\psi}\otimes q'^*(\cA^{\lambda}_G\ast K))=0
\end{equation}
\sssec{} 
\label{Sect_3.1.8_for_version2}
We claim that (\ref{complex_for_2.1.7}) identifies canonically with
\begin{equation}
\label{complex_for_2.1.8_left_to_right}
\RG_c(\Bun_G, K\otimes ((q'_!\ev_{s'}^*\cL_{\psi})\ast \cA^{\lambda}_G))
\end{equation}

 Indeed, to see this we may assume $K$ constructible. Then (\ref{complex_for_2.1.7}) identifies with
$$
\RG_c(\Bun_G, (\cA^{\lambda}_G\ast K)\otimes q'_! \ev_{s'}^*\cL_{\psi})\,\iso\,\DD\R\Hom(\cA^{\lambda}_G\ast K, q'_*\ev_{s'}^!\cL_{\psi^{-1}}[2])
$$
By (\cite{FGV}, 5.3.9), the functor $Shv(\Bun_G)\to Shv(\Bun_G), L\mapsto \cA^{\lambda}_G\ast L$ is left adjoint to the functor $L\mapsto L\ast \cA^{\lambda}_G$. So, (\ref{complex_for_2.1.7}) identifies with 
$$
\DD\R\Hom(K, (q'_*\ev_{s'}^!\cL_{\psi^{-1}}[2])\ast \cA^{\lambda}_G),
$$
which in turn identifies with (\ref{complex_for_2.1.8_left_to_right}) as desired.

\sssec{} Define the stack $\Bunb_U^{\cF'_T}$ as in \cite{FGV}. More precisely, if $[G,G]$ is simply-connected this is exactly the definition from \cite{FGV}. In this case it classifies $\cF'_G\in\Bun_G$ and a collection of nonzero maps over $X$
$$
\kappa^{\check{\lambda}}: \cL^{\check{\lambda}}_{\cF'_T}\hook{} \cV^{\check{\lambda}}_{\cF'_G}, \; \check{\lambda}\in\check{\Lambda}^+
$$
subject to the Pl\"ucker relations. Here $\check{\Lambda}^+$ is the set of dominant weights for $G$, and $\cV^{\check{\lambda}}$ denotes the Weyl module coresponding to $\check{\lambda}$ (cf. \cite{FGV}, Section~1.4). 

 Let also $_{y,\infty}\Bunb_U^{\cF'_T}$ be the version of $\Bunb_U^{\cF'_T}$, where the maps $\kappa^{\check{\lambda}}$ are allowed to have any poles at $y$. 

If $[G,G]$ is not simply-connected, one changes the definition of $\Bunb_U^{\cF'_T}$ as is explained in Schieder's paper (\cite{Sch}, Section~7) and similarly for $_{y,\infty}\Bunb_U^{\cF'_T}$. 
 
\sssec{}  As in (\cite{FGV}, Section~5.3), one defines the Hecke action of $\Perv_{G(\cO_y)}(\Gr_{G,y})$ on $Shv(_{y,\infty}\Bunb_U^{\cF'_T})$. So, for $\cS\in \Perv_{G(\cO_y)}(\Gr_{G,y})$, $\cT\in Shv(_{y,\infty}\Bunb_U^{\cF'_T})$ we get the objects $\cT\ast \cS$ and $\cS\ast\cT$ in $Shv(_{y,\infty}\Bunb_U^{\cF'_T})$. 

\sssec{} Let $j: \Bun_U^{\cF'_T}\to {_{y,\infty}\Bunb_U^{\cF'_T}}$ be the embedding.
Let $\bar q': {_{y,\infty}\Bunb_U^{\cF'_T}}\to\Bun_G$ be the projection. Since 
$$
\bar q'_!: Shv(_{y,\infty}\Bunb_U^{\cF'_T})\to Shv(\Bun_G)
$$ 
commutes with the actions of the Hecke functors at $y$, (\ref{complex_for_2.1.8_left_to_right}) identifies with   
\begin{equation}
\label{complex_for_referee_from_left_to_right}
\RG_c(\Bun_G, K\otimes \bar q'_!((j_!\ev_{s'}^*\cL_{\psi})\ast \cA^{\lambda}_G))\,\iso\, \RG_c({_{y,\infty}\Bunb_U^{\cF'_T}}, (\bar q')^*K\otimes ((j_!\ev_{s'}^*\cL_{\psi})\ast \cA^{\lambda}_G))
\end{equation}

\sssec{} For $\nu\in\Lambda$ let $S^{\nu}\subset \Gr_G$ denote the $U(F)$-orbit through $t^{\nu}\in\Gr_G$, here $F$ is our local field (cf. \cite{FGV}, Section~7.1.1).

\sssec{} For $\nu\in\Lambda$ let $_{y,\nu}\Bunt_U^{\cF'_T}\subset {_{y,\infty}\Bunb_U^{\cF'_T}}$ be the locally closed substack given by the property (in the case of $[G,G]$ simply-connected) that for any $\check{\lambda}\in\check{\Lambda}^+$, the map
$$
\cL^{\check{\lambda}}_{\cF'_T}(-\<\nu, \check{\lambda}\>y)\to \cV^{\check{\lambda}}_{\cF'_G}
$$
are regular over $X$ and have no zeros at $y$. (We leave it to a reader to define an analog of this stack for $G$ arbitrary reductive). 

 The stacks $_{y,\nu}\Bunt_U^{\cF'_T}$ for $\nu\in\Lambda$ form a stratification of $_{y,\infty}\Bunb_U^{\cF'_T}$.
 
\sssec{} Fix $\nu\in\Lambda$. It suffices to show that the contribution of the stratum $_{y,\nu}\Bunt_U^{\cF'_T}$ to integral (\ref{complex_for_referee_from_left_to_right}) vanishes.

 Set $\cF_T=\cF'_T(-\nu y)$. We have the open immersion $\Bun_U^{\cF_T}\hook{} {_{y,\nu}\Bunt_U^{\cF'_T}}$. By construction, the $*$-restriction of $(j_!\ev_{s'}^*\cL_{\psi})\ast \cA^{\lambda}_G$ to ${_{y,\nu}\Bunt_U^{\cF'_T}}$ is the extension by zero from $\Bun_U^{\cF_T}$. 

\sssec{} Let $Z$ be the stack classifying $(\cF'_B,\gamma')\in\Bun_U^{\cF'_T}$, for which we set $\cF'_G=\cF'_B\times_B G$, and $(\cF'_G, \cF_G,\beta)\in {_y\cH_G}$ such that $\cF_G$ is in the position $\le\lambda$ with respect to $\cF'_G$. So, 
$$
Z={_y\ov{\cH}^{-w_0(\lambda)}_G}\times_{\Bun_G} \Bun_U^{\cF'_T},
$$
where we used the map $h^{\ra}$ to form the fibred product. 

 Pick a trivialization $\epsilon: \cF'_T\mid_{D_y}\,\iso\,\cF^0_T\mid_{D_y}$, where $\cF^0_T$ denotes the trivial $T$-torsor. This gives rise to a $U(\cO_y)$-torsor denoted $\cU^{\epsilon}$ over $\Bun_U^{\cF'_T}$, it classifies trivializations $\cF'_B\,\iso\, \cF^0_B\mid_{D_y}$ inducing $\cF'_B\times_B T\toup{\epsilon\gamma'} \cF^0_T\mid_{D_y}$. 
 
  We have a projection $h^{\ra}_Z: Z\to \Bun_U^{\cF'_T}$ sending the above point to $(\cF'_B,\gamma')$. It realizes $Z$ as a fibration 
\begin{equation}
\label{fibration_for_Sect_2.1.8}
\cU^{\epsilon}\times^{U(\cO_y)} \ov{\Gr}_G^{\lambda}\,\iso\, Z
\end{equation}

 Let $Z^{\nu}\subset Z$ be the substack that identifies via $h^{\ra}_Z$ with 
$$
\cU^{\epsilon}\times^{U(\cO_y)} (\ov{\Gr}_G^{\lambda}\cap S^{\nu})
$$ 
 
 The stack $Z^{\nu}$ classifies $(\cF'_B,\gamma', \cF_G, \beta)\in Z$ such that $\cF'_B$ induces a $B$-structure on $\cF_G$, which we denote as $\cF_B$, and such that the isomorphism $\gamma': \cF_B\times_B T\,\iso\, \cF'_T\mid_{X-y}$ extends to an isomorphism $\gamma: \cF_B\times_B T\,\iso\, \cF_T$ on $X$.
  
 Let $h^{\la}_Z: Z^{\nu}\to \Bun_U^{\cF_T}$ be the map that sends the above point to $(\cF_B,\gamma)$. 
 
\sssec{} The $*$-restriction of $(j_!\ev_{s'}^*\cL_{\psi})\ast \cA^{\lambda}_G$ under $\Bun_U^{\cF_T}\hook{} {_{y,\infty}\Bunb_U^{\cF'_T}}$ identifies with
\begin{equation}
\label{complx_for_2.1.13}
(h^{\la}_Z)_!(\cA^{\lambda}_G\mid_{\ov{\Gr}_G^{\lambda}\cap S^{\nu}}\tboxtimes \ev_{s'}^*\cL_{\psi})^r,
\end{equation}
here $\cA^{\lambda}_G\mid_{\ov{\Gr}_G^{\lambda}\cap S^{\nu}}$ denotes the $*$-restriction.

 Consider the projection $\pr: \Lambda\to\Lambda_{G_{ad}}$. If for any $i\in\cI$, 
\begin{equation}
\label{condition_conductor_for_Sect_2.1.14}
\<cond(s')_y+\pr(\nu), \check{\alpha}_i\>\ge 0
\end{equation} 
then we get a point $(\cF_T, s)\in \cY^0_T$. Namely, for $i\in\cI$ the datum of $s'$ gives rise to an inclusion $s_{\check{\alpha}_i}: \cL^{\check{\alpha}}_{\cF_T}\hook{}\Omega$. We then have the morphism $\ev_s: \Bun_U^{\cF_T}\to \Ga$ defined along the same lines as $\ev_{s'}$.

\begin{Lm} 
\label{Lm_2.1.14}
If for any $i\in\cI$ the condition (\ref{condition_conductor_for_Sect_2.1.14}) holds then (\ref{complx_for_2.1.13}) is isomorphic to $
\ev_s^*\cL_{\psi}\otimes M$ for some constant complex $M\in Shv(\Spec k)$. Otherwise, 
 (\ref{complx_for_2.1.13}) vanishes.
\end{Lm}

 The proof of Lemma~\ref{Lm_2.1.14} is given in Section~\ref{Sect_2.1.16}. 

\sssec{End of the proof of Theorem~\ref{Th_preservation_Hecke} for $x$ regular} Let $q: \Bun_U^{\cF_T}\to\Bun_G$ be the natural map. By Lemma~\ref{Lm_2.1.14},
the contribution of the substack $\Bun_U^{\cF_T}$ in (\ref{complex_for_referee_from_left_to_right}) identifies with
$$
\RG_c(\Bun_U^{\cF_T}, q^*K\otimes \ev_s^*\cL_{\psi})
$$
tensored by some constant complex. The latter cohomology vanishes
by definition of $\cF_{\OO}$. We are done. \QED
 
\sssec{Proof of Lemma~\ref{Lm_2.1.14}}
\label{Sect_2.1.16} A detailed argument for any $x$ even is given in Section~\ref{Sect_3} below. Here we only give a sketch.

As in (\cite{Ga1}, Section~4), one defines the Whittaker category $Shv^W(_{y,\infty}\Bunb_U^{\cF'_T})$ corresponding to $s'$. Namely, first one defines a full Serre abelian subcategory 
$$
Shv^W(_{y,\infty}\Bunb_U^{\cF'_T})^{\heartsuit}\subset Shv(_{y,\infty}\Bunb_U^{\cF'_T})^{\heartsuit}
$$ 
It is singled out by an equivariance condition under some groupoid. More precisely, an object $K\in Shv(_{y,\infty}\Bunb_U^{\cF'_T})^{\heartsuit}$ by definition lies in $Shv^W(_{y,\infty}\Bunb_U^{\cF'_T})^{\heartsuit}$ if for any constructible perverse subsheaf $K'\subset K$, $K'$ is equivariant under the corresponding groipoid. Our appoach to defining $Shv^W(_{y,\infty}\Bunb_U^{\cF'_T})^{\heartsuit}$ is justified by
Remark~\ref{Rem_2.1.18} below.

Now $Shv^W(_{y,\infty}\Bunb_U^{\cF'_T})^{\heartsuit}$ is presentable and closed under colimits in $Shv(_{y,\infty}\Bunb_U^{\cF'_T})^{\heartsuit}$. Define 
\begin{equation}
\label{subcat_W_for_proof_of_Lm_3.1.14}
Shv^W(_{y,\infty}\Bunb_U^{\cF'_T})\subset Shv(_{y,\infty}\Bunb_U^{\cF'_T})
\end{equation} 
as the full subcategory of those objects whose all perverse cohomology sheaves lie in $Shv^W(_{y,\infty}\Bunb_U^{\cF'_T})^{\heartsuit}$. Then $Shv^W(_{y,\infty}\Bunb_U^{\cF'_T})\in\DGCat_{cont}$, this category is equipped with a t-structure induced from that on $Shv(_{y,\infty}\Bunb_U^{\cF'_T})$, so that the embedding (\ref{subcat_W_for_proof_of_Lm_3.1.14}) is continuous and t-exact.

The Hecke action of $\Perv_{G(\cO_y)}(\Gr_{G,y})$ on $Shv(_{y,\infty}\Bunb_U^{\cF'_T})$ preserves the full subcategory $Shv^W(_{y,\infty}\Bunb_U^{\cF'_T})$. 

For any locally closed substack $\cZ\subset {_{y,\infty}\Bunb_U^{\cF'_T}}$ stable under the corresponding groupoid, we may similarly define the version $Shv^W(\cZ)$ of the Whittaker category for $s'$. In particular, this holds for $\cZ=\Bun_U^{\cF_T}$, and we get $Shv^W(\Bun_U^{\cF_T})$. 

By the above, (\ref{complx_for_2.1.13}) is an object of $Shv^W(\Bun_U^{\cF_T})$. Our claim follows from Lemma~\ref{Lm_2.1.17} below.
\QED

\begin{Lm}
\label{Lm_2.1.17}
 If there is $i\in\cI$ such that (\ref{condition_conductor_for_Sect_2.1.14}) does not hold then $Shv^W(\Bun_U^{\cF_T})$ vanishes. Otherwise, any object of $Shv^W(\Bun_U^{\cF_T})$ is isomorphic to $\ev_s^*\cL_{\psi}\otimes M$ for some constant complex $M\in Shv(\Spec k)$. 
\end{Lm} 
\begin{proof}
This is analogous to (\cite{FGV}, Lemma~6.2.8). 
\end{proof}

\begin{Rem} 
\label{Rem_2.1.18}
Recall that for a classical algebraic stack locally of finite type, $\Perv(\cY)$ is defined as $(Shv(\cY)^{constr})^{\heartsuit}$. If $\cY$ is moreover of finite type, then the natural functor $\Ind(\Perv(\cY))\to Shv(\cY)^{\heartsuit}$ is an equivalence. However, this is not true in general for $\cY$ locally of finite type. In the latter case an object of $\Perv(\cY)$ is not necessarily compact in $Shv(\cY)^{\heartsuit}$. This happens for example for $\cY=\sqcup_{j\in\ZZ} \Spec k$.
\end{Rem}

\ssec{Minimal nilpotent orbit}

\sssec{} Assume $\gg$ simple. Let $\check{\alpha}$ be the highest root of $T$. Pick a nonzero element $x\in X_{\check{\alpha}}$ in the corresponding root subspace. Let $\OO=\OO_x$ be the nilpotent $G$-orbit through $x$. This is the minimal nilpotent orbit (\cite{CM}, Theorem~4.3.3). Let $P\subset G$ be the parabolic such that $\Lie P=\gp=\oplus_{i\ge 0}\gg_i$ is the Jacobson-Morozov parabolic subalgebra of $\gg$. It is known that $\gg_i=0$ for $\mid i\mid>2$. Recall that $U\subset P$ denotes the unipotent radical of $P$, and $V\subset U$ the subgroup with $\Lie V=\gg_2$ in this case. The group $P$ is called the Heisenberg parabolic of $G$. It is known that $U$ is a Heisenberg group with the center $V$, and $\dim\gg_2=1$. Let $M$ be the Levi of $P$ with $\Lie M=\gg_0$.
 
  Since $V$ is a 1-dimensional linear representation of $M$, the open $M\times\Gm$-orbit on $V^*$ is $V^*-\{0\}$.

\section{Case of even $x$}
\label{Sect_3}

\ssec{} In Section~\ref{Sect_3} we assume $x$ even and prove Theorem~\ref{Th_preservation_Hecke} in this case. Keep notations of Section~\ref{Sect_1.2}.  

\sssec{} If $[G,G]$ is simply connected, one has the stack $\Bunt_P$ defined in (\cite{BFGM}, Section~1). If $[G,G]$ is not simply-connected one modifies the definition of $\Bunt_P$ as is outlined in (\cite{Sch}, Section~7.4).   

\sssec{} 
\label{Sect_3.1.1}
Pick a maximal torus and Borel subgroups $T\subset B_M\subset M$  such that $H\in \Lie T$. This yields a Borel subgroup $B\subset P$, which is the preimage of $B_M$ under $P\to M$. Let $\cI_M\subset \cI$ be the subset corresponding to simple roots of $M$. Let $\Lambda_{G,P}$ denote the quotient of $\Lambda$ be the subgroup generated by $\alpha_i, i\in \cI_M$. This is the algebraic fundamental group of $M$. 
 Let $\check{\Lambda}_{G,P}$ be the lattice dual to $\Lambda_{G,P}$. Let $\check{\Lambda}^+_{G,P}\subset \Lambda_{G,P}$ be the subset of those weights, which are dominant for $B$.

\sssec{} Pick a closed point $y\in X$, our Hecke functors are applied at $y\in X$. Let $K\in\cF_{\OO}$, $\lambda\in\Lambda^+$. We want to show that $\cA^{\lambda}_G\ast K\in \cF_{\OO}$.

\sssec{} Pick a $k$-point of $\cY^0_M$ given by $(\cF'_M, s')$, where $s': (V/V')_{\cF'_M}\to\Omega$. We check that the $*$-fibre of $\Four(\gp_!\gq^*(\cA^{\lambda}_G\ast K))$ at this point vanishes. Our argument will be compatible with field extensions $k\to k'$ as in Section~\ref{Sect_Examples}, so this is sufficient.

 Recall that $U=V$. Let $\Bun_U^{\cF'_M}$ be the stack classifying $\cF'_P\in\Bun_P$ and an isomorphism $\gamma': \cF'_P\times_P M\,\iso\, \cF'_M$. Let $\Bun_{V/V'}^{\cF'_M}$ be the stack classifying a $P/V'$-torsor $\cF'_{P/V'}$ together with an isomorphism $\cF'_{P/V'}\times_{P/V'} M\,\iso\, \cF'_M$. So, $\Bun_{V/V'}^{\cF'_M}$ classifies exact sequences 
\begin{equation}
\label{seq_for_3.1.3}
0\to (V/V')_{\cF'_M} \to ?\to \cO_X\to 0
\end{equation} 
on $X$. 
 
 Let $\ev_{s'}: \Bun_U^{\cF'_M}\to\Ga$ be the composition $\Bun_U^{\cF'_M}\to \Bun_{V/V'}^{\cF'_M}\to\Ga$, where the second map is the pairing of (\ref{seq_for_3.1.3}) with $s'$. 
 
  Let $q': \Bun_U^{\cF'_M}\to\Bun_G$ be the natural map. It suffices to show that
\begin{equation}
\label{integral_one_for_Sect_3.1.3}
\RG_c(\Bun_U^{\cF'_M}, \ev_{s'}\cL_{\psi}\otimes q'^*(\cA^{\lambda}_G\ast K))=0
\end{equation}  

\sssec{} As in Section~\ref{Sect_3.1.8_for_version2}, the object (\ref{integral_one_for_Sect_3.1.3}) identifies canonically with
\begin{equation}
\label{Sect_complex_for_Sect_4.1.5_v2}
\RG_c(\Bun_G, K\otimes ((q'_!\ev^*_{s'}\cL_{\psi})\ast \cA^{\lambda}_G))
\end{equation}

\sssec{} 
\label{Sect_4.1.6_for_v2}
Let $\Bunt_U^{\cF'_M}$ denote the version of $\Bunt_P$, where we fix the $M$-torsor to be $\cF'_M$. If $[G,G]$ is simply-connected then, for a test scheme $S$ of finite type, an $S$-point of $\Bunt_U^{\cF'_M}$ is given by $\cF'_G$ on $X\times S$ and a collection of  maps of coherent sheaves 
$$
\kappa^{\cV}: (\cV^U)_{\cF'_M}\to \cV_{\cF'_G}
$$
on $X\times S$ for every $G$-module $\cV$, such that for any $k$-point $s\in S$, its restriction $\kappa^{\cV}_s$ to $X\times s$ is injective, and Pl\"ucker relations hold (\cite{BFGM}, Section~1.1). If $[G,G]$ is not simply-connected one corrects the above definition as in (\cite{Sch}, Section~7.4).

 Denote by $_{y,\infty}\Bunt_U^{\cF'_M}$ the version of $\Bunt_U^{\cF'_M}$, where the maps $\kappa^{\cV}$ are allowed to have any poles at $y$. If $[G,G]$ is not simply-connected, one defines $_{y,\infty}\Bunt_U^{\cF'_M}$ analogously.
 
\sssec{} As in (\cite{BG}, Section~4.1.4), one defines the Hecke action action of $\Perv_{G(\cO_y)}(\Gr_{G,y})$ on $Shv(_{y,\infty}\Bunt_U^{\cF'_M})$. For $\cS\in \Perv_{G(\cO_y)}(\Gr_{G,y})$, $\cT\in Shv(_{y,\infty}\Bunt_U^{\cF'_M})$ we get the corresponding objects $\cT\ast\cS$ and $\cS\ast\cT$ in $Shv(_{y,\infty}\Bunt_U^{\cF'_M})$. It is easy to see that the above Hecke action preserves the full subcategory 
$$
Shv(_{y,\infty}\Bunt_U^{\cF'_M})^{constr}
$$ 
of constructible objects (defined as in Section~\ref{Sect_1.1.2_conventions}). 

\sssec{} 
\label{Sect_4.1.8_for_v2}
Let $j: \Bun_U^{\cF'_M}\hook{}{_{y,\infty}\Bunt_U^{\cF'_M}}$ be the open embedding, $\bar q': {_{y,\infty}\Bunt_U^{\cF'_M}}\to \Bun_G$ be the projection. Since 
$$
\bar q'_!: Shv(_{y,\infty}\Bunt_U^{\cF'_M})\to Shv(\Bun_G)
$$
commutes with the actions of Hecke functors at $y$, (\ref{Sect_complex_for_Sect_4.1.5_v2}) identifies with
\begin{equation}
\label{complex_for_Sect_4.1.8_v2_for_referee}
\RG_c(\Bun_G, K\otimes \bar q'_!( (j_!\ev_{s'}^*\cL_{\psi})\ast \cA^{\lambda}_G))\,\iso\, \RG_c({_{y,\infty}\Bunt_U^{\cF'_M}}, (\bar q')^*K\otimes ( (j_!\ev_{s'}^*\cL_{\psi})\ast \cA^{\lambda}_G))
\end{equation}

\sssec{} 
\label{Sect_4.1.9_for_v2}
For $\theta\in\Lambda_{G,P}$ let 
$$
_{y,\theta}\Bunt_U^{\cF'_M}\subset {_{y,\infty}\Bunt_U^{\cF'_M}}
$$ 
be the locally closed substack defined as follows. For $[G,G]$ simply-connected it is given by the property that for any $\check{\lambda}\in\check{\Lambda}^+_{G, P}$ the map $\kappa^{\cV^{\check{\lambda}}}$ gives rise to a regular map
\begin{equation}
\label{map_for_check_lambda_in_Sect_4.1.9_v2}
\cL^{\check{\lambda}}_{\cF'_M}(-\<\theta, \check{\lambda}\>y)\to \cV^{\check{\lambda}}_{\cF'_G}
\end{equation}
over the whole of $X$, which moreover has no zeros in a neighbourhood of $y$. Recall that here $\cV^{\check{\lambda}}$ denotes the corresponding Weyl module for $G$. One adapts the above definition to the case of any $G$ reductive using the correction from (\cite{Sch}, Section~7).

 The stacks $_{y,\theta}\Bunt_U^{\cF'_M}$ for $\theta\in\Lambda_{G,P}$ form a stratification of ${_{y,\infty}\Bunt_U^{\cF'_M}}$. We claim that the contribution of each stratum $_{y,\theta}\Bunt_U^{\cF'_M}$ to the integral (\ref{complex_for_Sect_4.1.8_v2_for_referee}) vanishes.
 
  For $\theta\in\Lambda_{G,P}$ define the open substack $_{y,\theta}\Bun_U^{\cF'_M}\subset {_{y,\theta}\Bunt_U^{\cF'_M}}$ (in the case of $[G,G]$ simply-connected) by
the property that for any $\check{\lambda}\in\check{\Lambda}^+_{G, P}$ the maps (\ref{map_for_check_lambda_in_Sect_4.1.9_v2}) have no zeros on the whole of $X$. We leave it to a reader to adapt this definition for the general $G$ reductive.

 By construction, the $*$-restriction of $(j_!\ev_{s'}^*\cL_{\psi})\ast \cA^{\lambda}_G$ to 
$_{y,\theta}\Bunt_U^{\cF'_M}$ is the extension by zero from ${_{y,\theta}\Bun_U^{\cF'_M}}$. Write
$$
i_{\theta}: {_{y,\theta}\Bun_U^{\cF'_M}}\hook{} {_{y,\infty}\Bunt_U^{\cF'_M}}
$$
for the natural inclusion.

\sssec{} 
\label{Sect_3.1.6}
For $\theta\in \Lambda_{G,P}$ denote by $\Gr_M^{\theta}$ be the connected component of $\Gr_M$ containing the element $t^{\lambda}M(\cO)$ for any $\lambda\in\Lambda$ over $\theta$. Let $\Gr_P^{\theta}$ be the preimage of $\Gr_M^{\theta}$ under the natural map $\Gr_P\to \Gr_M$. 

  We have a natural map $\Gr_P^{\theta}\to\Gr_G$. At the level of reduced finite-dimensional subschemes of $\Gr_G$, for any closed subscheme of finite type $S\subset \Gr_G$ this gives a stratification of $S$ by $S\cap \Gr_P^{\theta}$, $\theta\in\Lambda_{G,P}$.
  
 
\sssec{} 
\label{Sect_3.1.7}
As in \cite{BG} write $\Gr_M^+\subset \Gr_M$ for the positive part of the affine grassmanian for $M$. It classifies an $M$-torsor $\cF_M$ on $X$ with a trivialization 
$\beta_M: \cF_M\,\iso\, \cF^0_M\mid_{X-y}$ such that for any finite-dimensional $G$-module $V$ the map
$$
\beta_M: V^{U}_{\cF_M}\to V^{U}_{\cF^0_M}
$$
is regular over $X$. For $\theta\in\Lambda_{G,P}$ set 
$$
\Gr_M^{\theta, +}=\Gr_M^{\theta}\cap \Gr_M^+
$$

Let $\Gr_M(\cF'_M)$ be the stack classifying $\cF_M\in\Bun_M$ and an isomorphism $\beta_M: \cF_M\,\iso\, \cF'_M\mid_{X-y}$. A trivialization $\epsilon_M: \cF'_M\,\iso\, \cF^0_M\mid_{D_y}$ yields $\Gr_M(\cF'_M)\,\iso\,\Gr_M$. Let 
$$
\Gr_M^{\theta}(\cF'_M)\subset \Gr_M(\cF'_M)
$$ 
be the substack corresponding to $\Gr_M^{\theta}$ under this identification. It is independent of a choice of $\epsilon_M$. One defines $\Gr_M^+(\cF'_M), \Gr_M^{\theta, +}(\cF'_M)$ similarly. 

\sssec{} 
\label{Sect_4.1.12_for_v2}
Let $\Gr_M^{\theta, -}(\cF'_M)$ be the stack classifying a $M$-torsor $\cF_M$ on $X$, an isomorphism $\beta_M:  \cF_M\,\iso\,\cF'_M\mid_{X-y}$ such that 
$$
(\cF'_M,\beta_M)\in \Gr_M^{-\theta, +}(\cF_M)
$$ 

For $\theta\in\Lambda_{G,P}$ the stack $_{y,\theta}\Bun_U^{\cF'_M}$ classifies a $P$-torsor $\cF_P$ on $X$ for which we set $\cF_M=\cF_P\times_P M$, and an isomorphism $\beta_M: \cF_M\,\iso\,\cF'_M\mid_{X-y}$ such that $(\cF_M,\beta_M)\in \Gr_M^{\theta, -}(\cF'_M)$. Let 
$$
q_{\theta}: {_{y,\theta}\Bun_U^{\cF'_M}}\to \Gr_M^{\theta, -}
(\cF'_M)
$$ 
be the projection sending the above point to $(\cF_M,\beta_M)$.
    
\sssec{} 
\label{Sect_3.1.11}
Pick $\theta\in\Lambda_{G,P}$. Pick a $k$-point $\eta=(\cF_M, \beta_M)\in \Gr_M^{\theta, -}(\cF'_M)$. The fibre of $q_{\theta}$ over $\eta$ identifies with $\Bun_U^{\cF_M}$. Write 
\begin{equation}
\label{inclusion_iota_eta_for_Sect_4.1.13_v2}
\iota_{\eta}: \Bun_U^{\cF_M}\hook{} {_{y,\theta}\Bun_U^{\cF'_M}}
\end{equation} 
for the corresponding closed immersion.

 Say that $\eta$ has \select{positive $s'$-conductor} if the map 
$$
(V/V')_{\cF_M}\toup{\beta_M} (V/V')_{\cF'_M}\toup{s'}\Omega\mid_{X-y}
$$ 
initially defined over $X-y$ extends to a regular map $(V/V')_{\cF_M}\toup{s} \Omega$ on $X$. If $\eta$ has positive $s'$-conductor then one gets a morphism $\ev_s: \Bun_U^{\cF_M}\to\Ga$ defined along the same lines as $\ev_{s'}$. 

\begin{Lm} 
\label{Lm_3.1.12}
If $\eta$ has positive $s'$-conductor then 
$$
\iota_{\eta}^*i_{\theta}^*((j_!\ev_{s'}^*\cL_{\psi})\ast \cA^{\lambda}_G)
$$
is isomorphic to $\ev_s^*\cL_{\psi}\otimes M$ for some constant complex $M\in Shv(\Spec k)$. Otherwise, it vanishes.
\end{Lm}
 
 The proof of Lemma~\ref{Lm_3.1.12} is given in Section~\ref{Sect_3.2.11}.
 
\sssec{End of the proof of Theorem~\ref{Th_preservation_Hecke} for $x$ even} 
\label{Sect_3.1.13}
Recall that we picked $\theta\in\Lambda_{G,P}$ and a $k$-point $\eta=(\cF_M, \beta_M)\in \Gr_M^{\theta, -}(\cF'_M)$. 
It suffices to show that 
$$
(q_{\theta})_!\left(
i_{\theta}^*(\bar q')^*K\otimes  i_{\theta}^*((j_!\ev_{s'}^*\cL_{\psi})\ast \cA^{\lambda}_G)\right)
$$
vanishes. For this in turn it suffices to show that the $*$-fibre of the latter complex at $\eta$ vanishes (our argument is compatible with field extensions). Let $q: \Bun_U^{\cF_M}\to \Bun_G$ be the natural map. We must show that
$$
\RG_c(\Bun_U^{\cF_M}, q^*K\otimes \iota_{\eta}^*i_{\theta}^*((j_!\ev_{s'}^*\cL_{\psi})\ast \cA^{\lambda}_G))=0
$$
By Lemma~\ref{Lm_3.1.12} we may assume that $\eta$ has positive $s'$-conductor. Then our claim follows from Lemma~\ref{Lm_3.1.12}, because
$$
\RG_c(\Bun_U^{\cF_M}, q^*K\otimes\ev_s^*\cL_{\psi})=0
$$ 
by definition of $\cF_{\OO}$. Indeed, $s$ and $s'$ coincide at the generic point of $X$, so $(\cF_M, s)\in \cY^0_M$. Theorem~\ref{Th_preservation_Hecke} is proved for $x$ even. \QED

\ssec{Definition of the $W$-category} 
\label{Sect_3.2}

\sssec{} 
\label{Sect_3.2.1}
As in (\cite{Ga1}, Section~4) or (\cite{Ga2}, Section~2), we define a full $\DG$-subcategory $Shv^W(_{y,\infty}\Bunt_U^{\cF'_M})\subset Shv(_{y,\infty}\Bunt_U^{\cF'_M})$ attached to $s'$. Our strategy is first to define a full Serre abelian subcategory $Shv^W(_{y,\infty}\Bunt_U^{\cF'_M})^{\heartsuit}$ of the heart $Shv(_{y,\infty}\Bunt_U^{\cF'_M})^{\heartsuit}$. Then we let $Shv^W(_{y,\infty}\Bunt_U^{\cF'_M})$ be the full subcategory of those objects of $Shv(_{y,\infty}\Bunt_U^{\cF'_M})$ whose all perverse cohomology sheaves lie in $Shv^W(_{y,\infty}\Bunt_U^{\cF'_M})^{\heartsuit}$.

 We give the definition under the assumption that $[G,G]$ is simply-connected, the general case is done similarly adding the correction from (\cite{Sch}, Section~7). 
  
\sssec{} 
\label{Sect_4.2.2_now}
Let $z\in X$ be a closed point different from $y$. Let $(V/V')^{reg}_z$ (resp., $(V/V')^{mer}_z$) denote the group scheme (resp., group ind-scheme) of sections of $(V/V')_{\cF'_M}$ over $D_z$ (resp., over $D^*_z$). Here `reg' stands for `regular', and `mer' stands for `meromorphic'. 

 More generally, let $\bar z=\{z_1,\ldots, z_m\}$ be a finite collection of pairwise different closed points of $X-y$. Let $D_{\bar z}$ be the formal neighbourhood of $\bar z=\cup_{i=1}^m z_i$. Replacing $D_z$ by $D_{\bar z}\,\iso\, \prod_i D_{z_i}$, one 
similarly defines $(V/V')^{reg}_{\bar z}$, $(V/V')^{mer}_{\bar z}$. 

 For $1\le i\le m$ let $\chi_{z_i}: (V/V')^{mer}_{z_i}\to\Ga$ be the character given as the composition
$$
(V/V')^{mer}_{z_i}\toup{s'}\Omega(F_{z_i})\toup{\Res}\Ga
$$
Here $F_{z_i}$ is the completion of the field of rational functions on $X$ at $z_i$. 
Define the character 
$$
\chi_{\bar z}: (V/V')^{mer}_{\bar z}\to\Ga
$$ 
as $\sum_{i=1}^m \chi_{z_i}$. Note that 
$$
(V/V')^{mer}_{\bar z}/(V/V')^{reg}_{\bar z}
$$ 
can be seen as the affine grassmanian $\Gr_{(V/V')_{\cF'_M}, \bar z}$ for the group scheme $(V/V')_{\cF'_M}$ at $\bar z$. The latter classifies $(V/V')_{\cF'_M}$-torsors on $D_{\bar z}$ together with a trivialization over $D_{\bar z}^*$. 

 Since $\chi_{\bar z}$ is trivial on $(V/V')^{reg}_{\bar z}$, $\chi_{\bar z}$ writes as the composition
$$
(V/V')^{mer}_{\bar z}\to \Gr_{(V/V')_{\cF'_M}, \bar z}\;\toup{\bar\chi_{\bar z}\;\,}\Ga
$$
 
\sssec{} Let 
$$
_{y,\infty}\Bunt_{U, \,{\rm{good\; at}}\; \bar z}^{\cF'_M}\subset {_{y,\infty}\Bunt_U^{\cF'_M}}
$$ 
be the open substack given by the property that all the maps 
$$
\kappa^{\cV}: \cV^U_{\cF'_M}\hook{} \cV_{\cF'_G}(\infty y)
$$
are maximal over $D_{\bar z}$. For a point of $_{y,\infty}\Bunt_{U, \,{\rm{good\; at}}\; \bar z}^{\cF'_M}$ we get a $P$-torsor $\cF'_P$ on $D_{\bar z}$ with an isomorphism $\gamma': \cF'_P\times_P M\,\iso\, \cF'_M$ over $D_{\bar z}$.
 
\sssec{} 
\label{Sect_4.2.4_now}
Let $\cG_{\bar z}$ denote the stack classifying two points of $_{y,\infty}\Bunt_{U, \,{\rm{good\; at}}\;  \bar z}^{\cF'_M}$ denoted by $(\cF_G, \kappa^{\cV})$ and $(\cF'_G, \kappa'^{\cV})$, and an isomorphism $\beta_{\bar z}: \cF_G\,\iso\, \cF'_G\mid_{X-\cup_i z_i}$ making the diagram commutative 
$$
\begin{array}{ccc}
(\cV^U)_{\cF'_M} & \toup{\kappa^{\cV}} & \cV_{\cF_G}(\infty y)\\
& \searrow\lefteqn{\scriptstyle \kappa'^{\cV}} & \downarrow\lefteqn{\scriptstyle \beta_{\bar z}}\\
&& \cV_{\cF_G}(\infty y)
\end{array}
$$
for any finite-dimensional $G$-module $\cV$. 

 For a point of $\cG_{\bar z}$ we get two $P$-torsors $\cF_P$ and $\cF'_P$ on $D_{\bar z}$, and $\beta_{\bar z}$ gives rise to an isomorphism $\beta_{P,\bar z}: \cF_P\,\iso\, \cF'_P\mid_{D^*_{\bar z}}$ making the diagram commute
$$
\begin{array}{ccc}
\cF_P\times_P M &\toup{\gamma} & \cF'_M\\
\downarrow\lefteqn{\scriptstyle \bar\beta_{P,\bar z}}  & 
\nearrow\lefteqn{\scriptstyle \gamma'}\\
\cF'_P\times_P M
\end{array}
$$ 
Here $\bar\beta_{P, \bar z}$ is the extension of scalars of $\beta_{P,\bar z}$ under $P\to M$. So, $\bar\beta_{P, \bar z}$ extends to a regular map over $D_{\bar z}$. 

\sssec{} For a point of $\cG_{\bar z}$ as above let $\cN$ be the sheaf of isomorphisms $\cF_P\to \cF'_P$ of  $P$-torsors over $D_{\bar z}$ compatible with $\bar\beta_{P,\bar z}$. The sheaf of automorphisms of $\cF_P$ on $D_{\bar z}$ acting trivially on $\cF_P\times_P M$ is $V_{\cF_P}$ over $D_{\bar z}$. This sheaf acts on $\cN$ via its action on $\cF_P$. This way $\cN$ becomes a torsor under $V_{\cF_P}$ over $D_{\bar z}$ together with a trivialization $\beta_{\cN}$ over $D^*_{\bar z}$. Namely, $\beta_{P,\bar z}$ gives the corresponding trivialization. The extension of scalars of $(\cN, \beta_{\cN})$ via $V_{\cF_P}\to (V/V')_{\cF'_M}$ gives a morphism 
$$
\cG_{\bar z}\to \Gr_{(V/V')_{\cF'_M}, \bar z}
$$

 Define $\chi_{\cG}: \cG_{\bar z}\to\Ga$ as the composition
$$
\cG_{\bar z}\to \Gr_{(V/[V,V])_{\cF'_M}, \bar z}\;\toup{\bar\chi_{\bar z}\;\,} \Ga
$$ 

\sssec{} 
\label{Sect_4.2.6_now}
We have a diagram of projections 
$$
{_{y,\infty}\Bunt_{U, \,{\rm{good\; at}}\; \bar z}^{\cF'_M}}\;\getsup{h^{\la}_{\cG}}\;
 \cG_{\bar z}\;\toup{h^{\ra}_{\cG}}\; {_{y,\infty}\Bunt_{U, \,{\rm{good\; at}}\; \bar z}^{\cF'_M}}
$$
Here $h^{\ra}_{\cG}$ (resp., $h^{\la}_{\cG}$) sends the above point to $(\cF'_G, \kappa'^{\cV})$ (resp., to $(\cF_G, \kappa^{\cV})$). This way $\cG_{\bar z}$ gets a structure of a groupoid over $_{y,\infty}\Bunt_{U, \,{\rm{good\; at}}\; \bar z}^{\cF'_M}$.

\sssec{} 
\label{Sect_4.2.7_now}
We first define 
$$
Shv^{W, constr, \heartsuit}_{{\rm{good\; at}}\; \bar z}\subset (Shv(_{y,\infty}\Bunt_{U, \,{\rm{good\; at}}\; \bar z}^{\cF'_M})^{constr})^{\heartsuit}
$$ 
as the full subcategory of those objects which are 
 $(\cG_{\bar z}, \chi_{\cG}^*\cL_{\psi})$-equivariant. The precise sense of this is as in (\cite{Ga1}, Section~4.7). 

 Let now $\bar z'$ be another such collection of points and $\bar z''=\bar z\cup \bar z'$. Then we may consider the versions of the above objects for $\bar z'$ and $\bar z''$. 
As in (\cite{Ga2}, Section~2.5), if $\bar z$ is not empty then for an object 
$$
\cK\in  (Shv(_{y,\infty}\Bunt_{U, \,{\rm{good\; at}}\; \bar z''}^{\cF'_M})^{constr})^{\heartsuit}$$ 
the 
$(\cG_{\bar z}, \chi_{\cG}^*\cL_{\psi})$-equivariance and $(\cG_{\bar z''}, \chi_{\cG}^*\cL_{\psi})$-equivariance are equivalent. 

\begin{Def} 
\label{Def_3.2.6}
Let 
$$
Shv^{W, constr,\heartsuit}\subset (Shv(_{y,\infty}\Bunt_U^{\cF'_M})^{constr})^{\heartsuit}
$$
be the full subcategory of $\cK$ such that for any nonempty finite collection $\bar z$ as above, the restriction of $\cK$ to $(Shv(_{y,\infty}\Bunt_{U, \,{\rm{good\; at}}\; \bar z}^{\cF'_M})^{constr})^{\heartsuit}$ lies in $Shv^{W, constr, \heartsuit}_{{\rm{good\; at}}\; \bar z}$.
\end{Def}
     
 One checks that $Shv^{W, constr,\heartsuit}$ is an abelian Serre subcategory of $(Shv(_{y,\infty}\Bunt_U^{\cF'_M})^{constr})^{\heartsuit}$. 
  
\begin{Def} 
\label{Def_3.2.7}
Let 
$$
Shv^W(_{y,\infty}\Bunt_U^{\cF'_M})^{\heartsuit}\subset   Shv(_{y,\infty}\Bunt_U^{\cF'_M})^{\heartsuit}
$$
be the full subcategory of those $\cK$ such that for any $\cK'\in (Shv(_{y,\infty}\Bunt_U^{\cF'_M})^{constr})^{\heartsuit}$ and an embedding $\cK'\hook{}\cK$ one has $\cK'\in Shv^{W, constr,\heartsuit}$.
\end{Def}

 One checks that $Shv^W(_{y,\infty}\Bunt_U^{\cF'_M})^{\heartsuit}$ is a Serre abelian subcategory of $Shv(_{y,\infty}\Bunt_U^{\cF'_M})^{\heartsuit}$ closed under small colimits. 
\begin{Def} 
\label{Def_3.2.8}
Let  
\begin{equation}
\label{W-subcategory_Sect_3.2.7}
Shv^W(_{y,\infty}\Bunt_U^{\cF'_M})\subset   Shv(_{y,\infty}\Bunt_U^{\cF'_M})
\end{equation}
be the full subcategory of those objects $\cK$ whose all perverse cohomology sheaves lie in 
$Shv^W(_{y,\infty}\Bunt_U^{\cF'_M})^{\heartsuit}$.
\end{Def}

 We conclude that $Shv^W(_{y,\infty}\Bunt_U^{\cF'_M})\in\DGCat_{cont}$, the inclusion (\ref{W-subcategory_Sect_3.2.7}) is continuous, and $Shv^W(_{y,\infty}\Bunt_U^{\cF'_M})$ is compatible with the perverse t-structure on $Shv(_{y,\infty}\Bunt_U^{\cF'_M})$, so inherits a t-structure (compatible with filtered colimits).
 
 Let also 
$$
Shv^W(_{y,\infty}\Bunt_U^{\cF'_M})^{constr}=Shv(_{y,\infty}\Bunt_U^{\cF'_M})^{constr} \cap Shv^W(_{y,\infty}\Bunt_U^{\cF'_M})
$$
Then $Shv^W(_{y,\infty}\Bunt_U^{\cF'_M})^{constr}\in\DGCat^{non-cocmpl}$, and this $\DG$-category inherits a t-structure from $Shv(_{y,\infty}\Bunt_U^{\cF'_M})^{constr}$.
 
Since in our definitions $\bar z$ are chosen distinct from $y$, one gets the following.

\begin{Pp} 
\label{Pp_3.2.9}
The Hecke action of $\Perv_{G(\cO_y)}(\Gr_{G,y})$ on $Shv(_{y,\infty}\Bunt_U^{\cF'_M})$ preserves the full subcategories $Shv^W(_{y,\infty}\Bunt_U^{\cF'_M})$ and $Shv^W(_{y,\infty}\Bunt_U^{\cF'_M})^{constr}$. \QED
\end{Pp}

\sssec{Proof of Lemma~\ref{Lm_3.1.12}} 
\label{Sect_3.2.11}
For $\theta\in\Lambda_{G,P}$ one similarly defines the versions 
$$
Shv^W(_{y,\theta}\Bun_U^{\cF'_M}),\;\;\;\; Shv^W(_{y,\theta}\Bun_U^{\cF'_M})^{constr}
$$ 
of $W$-categories. As soon as a closed point $\eta=(\cF_M, \beta_M)\in \Gr_M^{\theta, -}(\cF'_M)$ is picked, the closed immersion (\ref{inclusion_iota_eta_for_Sect_4.1.13_v2}) is stable under the action of the corresponding groupoids. So, we similarly get the categories $Shv^W(\Bun_U^{\cF_M})$, $Shv^W(\Bun_U^{\cF_M})^{constr}$. By functoriality, $\iota_{\eta}^*i_{\theta}^*$ sends 
$Shv^W(_{y,\infty}\Bunt_U^{\cF'_M})$ to $Shv^W(\Bun_U^{\cF_M})$ and preserves constructibility. Our claim is reduced to Lemma~\ref{Lm_3.2.12} below. \QED

\begin{Lm}
\label{Lm_3.2.12}
 If $\eta$ has positive $s'$-conductor then any object of $Shv^W(\Bun_U^{\cF_M})^{constr}$ is isomorphic to $\ev_s^*\cL_{\psi}\otimes M$ for some constant complex $M\in Shv(\Spec k)$. Otherwise, $Shv^W(\Bun_U^{\cF_M})^{constr}$ vanishes.
\end{Lm}
\begin{proof}
This is analogous to (\cite{FGV}, Lemma~6.2.8).
\end{proof}

Theorem~\ref{Th_preservation_Hecke} is proved for $x$ even. 

\section{Case of $x$ odd}
\label{Sect_x_odd}

\ssec{} In Section~\ref{Sect_x_odd} we assume $x$ odd and prove Theorem~\ref{Th_preservation_Hecke} in this case. Keep notations of Section~\ref{Sect_1.2}.

\sssec{} Recall that we picked a closed point $y\in X$, and our Hecke functors are applied at $y\in X$. Let $K\in\cF_{\OO}$, $\lambda\in\Lambda^+$. We want to show that $\cA^{\lambda}_G\ast K\in \cF_{\OO}$.

\sssec{} Pick a $k$-point of $\cY_P^0$ given by $\cF^1\in\Bun_{P/V'}$ and $s': (V/V')_{\cF'_M}\to\Omega$, where $\cF'_M=\cF^1\times_{P/V'} M$. We will show that the $*$-fibre of $\Four_P(\gp_!\gq^*(\cA^{\lambda}_G\ast K))$ at this point vanishes (our argument will be compatible with field extensions as in the previous sections).

Let $\cX_{\cF'_M}$ be the stack classifying exact sequences on $X$
\begin{equation}
\label{sequence_one_for_2.3.2}
0\to (V/V')_{\cF'_M}\to ?\to \cO_X\to 0
\end{equation}  
Let $\ev_{s'}: \cX_{\cF'_M}\to\A^1$ be the map sending (\ref{sequence_one_for_2.3.2}) to its pairing with $s'$.   

 Denote by $\Bun_U^{\cF'_M}$ the stack classifying a $P$-torsor $\cF'_P$ on $X$ and an isomorphism of $M$-torsors $\gamma': \cF'_P\times_P M\,\iso\,\cF'_M$ on $X$. 
 
 Recall the map $\act: \cX_P\to\Bun_{P/V'}$ defined in Section~\ref{Sect_1.2.6}. Let $\cX_1$ be the stack classifying an exact sequence (\ref{sequence_one_for_2.3.2}), for which we get the the $P/V'$-torsor $\cF'_{P/V'}=\act(\cF^1, (\ref{sequence_one_for_2.3.2}))$ on $X$, and a lifting of $\cF'_{P/V'}$ to a $P$-torsor $\cF'_P$.  Note that $\cF'_P$ is equipped with $\gamma': \cF'_P\times_P M\,\iso\, \cF'_M$. 
  
 Write $\act_1: \cX_1\to \Bun_U^{\cF'_M}$ for the map sending the above point to $(\cF'_P, \gamma')$. Let $\ev_1: \cX_1\to\Ga$ be the map sending the above point to the pairing of 
(\ref{sequence_one_for_2.3.2}) with $s'$. 

 In fact, $\cX_1$ classifies a $P$-torsor $\cF'_P$ with an isomorphism 
$$
\cF'_P\times_P (P/V)\,\iso\, \cF^1\times_{P/V'} (P/V)
$$
of $P/V$-torsors. 

 Write $\Bun_{U/V}^{\cF'_M}$ for the stack classifying a $P/V$-torsor $\cF'_{P/V}$ together with an isomorphism $\cF'_{P/V}\times_{P/V} M\,\iso\cF'_M$. So, $\Bun_{U/V}^{\cF'_M}$ classifies exact sequences 
$$
0\to (U/V)_{\cF'_M}\to ?\to \cO_X\to 0
$$ 
on $X$. We get a cartesian square
$$
\begin{array}{ccc}
\cX_1 & \toup{\act_1} & \Bun_U^{\cF'_M}\\
\downarrow &&\downarrow\\
\Spec k & \toup{\cF^1_{P/V}} & \Bun_{U/V}^{\cF'_M},
\end{array}
$$
where $\cF^1_{P/V}=\cF^1\times_{P/V'} (P/V)$.

\sssec{} Set
$$
\cE=(\act_1)_!\ev_1^*\cL_{\psi}
$$         

 Let $q': \Bun_U^{\cF'_M}\to\Bun_G$ be the natural map. For $p>0$ it suffices to show that 
\begin{equation}
\label{integral_for_Sect_2.3.2}
\RG_c(\Bun_U^{\cF'_M}, \cE\otimes (q')^*(\cA^{\lambda}_G\ast K))=0
\end{equation}

\sssec{} The complex (\ref{integral_for_Sect_2.3.2}) identifies canonically with
$$
\RG_c(\Bun_G, (\cA^{\lambda}_G\ast K)\otimes q'_!\cE)
$$
Now as in Section~\ref{Sect_3.1.8_for_version2}, one identifies (\ref{integral_for_Sect_2.3.2}) canonically with
\begin{equation}
\label{complex_for_Sect_5.1.3_v2}
\RG_c(\Bun_G, K\otimes ((q'_!\cE)\ast \cA^{\lambda}_G))
\end{equation}

\sssec{} Let the stacks $\Bunt_U^{\cF'_M}$, $_{y,\infty}\Bunt_U^{\cF'_M}$ be as in Section~\ref{Sect_4.1.6_for_v2}. As in Section~\ref{Sect_4.1.8_for_v2}, we have the maps 
$$
j: \Bun_U^{\cF'_M}\to {_{y,\infty}\Bunt_U^{\cF'_M}}, \;\;\;\; \bar q': {_{y,\infty}\Bunt_U^{\cF'_M}}\to \Bun_G,
$$ 
and (\ref{complex_for_Sect_5.1.3_v2}) identifies canonically with
\begin{equation}
\label{complex_for_Sect_5.1.4_v2}
\RG_c(\Bun_G, K\otimes \bar q'_!((j_!\cE)\ast \cA^{\lambda}_G))\,\iso\,\RG_c({_{y,\infty}\Bunt_U^{\cF'_M}}, (\bar q')^*K\otimes((j_!\cE)\ast \cA^{\lambda}_G))
\end{equation}

\sssec{} Recall the stratification of ${_{y,\infty}\Bunt_U^{\cF'_M}}$ by locally closed substacks $_{y,\theta}\Bunt_U^{\cF'_M}$ for $\theta\in\Lambda_{G,P}$ from Section~\ref{Sect_4.1.9_for_v2}. We claim that the contribution of each stratum $_{y,\theta}\Bunt_U^{\cF'_M}$ to (\ref{complex_for_Sect_5.1.4_v2}) vanishes.

 By construction, the $*$-restriction of $(j_!\cE)\ast \cA^{\lambda}_G$ to $_{y,\theta}\Bunt_U^{\cF'_M}$ is the extension by zero from $_{y, \theta}\Bun_U^{\cF'_M}$. So, we must show that for each $\theta\in\Lambda_{G,P}$,
$$
\RG_c({_{y, \theta}\Bun_U^{\cF'_M}}, i_{\theta}^*(\bar q')^*K\otimes i_{\theta}^*((j_!\cE)\ast \cA^{\lambda}_G))=0
$$  
 
  Recall the map $q_{\theta}: {_{y, \theta}\Bun_U^{\cF'_M}}\to \Gr_M^{\theta, -}(\cF'_M)$ from Section~\ref{Sect_4.1.12_for_v2}. It suffices to show that
\begin{equation}
\label{complex_for_Sect_5.1.6_for_v2}
(q_{\theta})_!i_{\theta}^*\left((\bar q')^*K\otimes ((j_!\cE)\ast \cA^{\lambda}_G)\right)=0 
\end{equation}

 This will be done in Section~\ref{Sect_End_proof_x_odd_v2} after introducing a suitable version of the $W$-category.

 

\ssec{Definition of the $W$-category}
\label{Sect_4.2}

\sssec{} Define the full $\DG$-subcategory $Shv^W(_{y,\infty}\Bunt_U^{\cF'_M})\subset Shv(_{y,\infty}\Bunt_U^{\cF'_M})$ attached to $s'$ as in Section~\ref{Sect_3.2} with some changes. Here are the details.
 
As in Section~\ref{Sect_4.2.2_now}, for a collection of pairwise distinct closed points $\bar z=\{z_1,\ldots, z_m\}\subset X-y$ we have the same objects 
$$
(V/V')^{reg}_{\bar z}, \;\;\;\;\; (V/V')^{mer}_{\bar z}\!, \;\;\;\;\; \chi_{\bar z}: (V/V')^{mer}_{\bar z}\to\Ga
$$ 
and
$$
\bar\chi_{\bar z}: \Gr_{(V/V')_{\cF'_M}, \bar z}\to \Ga
$$
The open substack 
$$
_{y,\infty}\Bunt_{U, \,{\rm{good\; at}}\; \bar z}^{\cF'_M}\subset {_{y,\infty}\Bunt_U^{\cF'_M}}
$$
is defined as in Section~\ref{Sect_3.2}.

\sssec{} The definition of the groupoid $\cG_{\bar z}$ is changed compared to Section~\ref{Sect_4.2.4_now} as follows.

 Let $\cG_{\bar z}$ be the ind-stack classifying two points of $_{y,\infty}\Bunt_{U, \,{\rm{good\; at}}\; \bar z}^{\cF'_M}$, namely $(\cF_G, \kappa^{\cV})$ and $(\cF'_G, \kappa'^{\cV})$, and an isomorphism $\beta_{\bar z}: \cF_G\,\iso\cF'_G\mid_{X-\cup_i z_i}$ subject to the following conditions.

 First, for any finite-dimensional $G$-module $\cV$ the diagram commutes
$$
\begin{array}{ccc}
(\cV^U)_{\cF'_M} & \toup{\kappa^{\cV}} & \cV_{\cF_G}(\infty y)\\
& \searrow\lefteqn{\scriptstyle \kappa'^{\cV}} & \downarrow\lefteqn{\scriptstyle \beta_{\bar z}}\\
&& \cV_{\cF_G}(\infty y)
\end{array}
$$
So, for a point of $\cG_{\bar z}$ we get two $P$-torsors $\cF_P$ and $\cF'_P$ on $D_{\bar z}$, and $\beta_{\bar z}$ gives rise to an isomorphism $\beta_{P,\bar z}: \cF_P\,\iso\, \cF'_P\mid_{D^*_{\bar z}}$ making the diagram commute
$$
\begin{array}{ccc}
\cF_P\times_P M &\toup{\gamma} & \cF'_M\\
\downarrow\lefteqn{\scriptstyle \bar\beta_{P,\bar z}}  & 
\nearrow\lefteqn{\scriptstyle \gamma'}\\
\cF'_P\times_P M
\end{array}
$$ 
Here $\bar\beta_{P,\bar z}$ is the extension of scalars of $\beta_{P, \bar z}$ under $P\to M$.
So, $\bar\beta_{P,\bar z}$ extends to a regular map over $D_{\bar z}$. 

 Second, we require that the extension of scalars
$$
\cF_P\times_P P/V\to \cF'_P\times_P P/V
$$
of $\beta_{P,\bar z}$ initially defined on $D_{\bar z}^*$ extends to an isomorphism of $P/V$-torsors over $D_{\bar z}$ denoted
$$
\beta_{P/V,\,\bar z}: \cF_P\times_P P/V\,\iso\, \cF'_P\times_P P/V\mid_{D_{\bar z}}
$$

\sssec{} For a point of $\cG_{\bar z}$ let $\cN$ be the sheaf of isomorphisms $\cF_P\to \cF'_P$ of  $P$-torsors over $D_{\bar z}$ compatible with $\beta_{P/V,\, \bar z}$. The sheaf of automorphisms of $\cF_P$ on $D_{\bar z}$ acting trivially on $\cF_P\times_P P/V$ is $V_{\cF_P}$ over $D_{\bar z}$. This sheaf acts on $\cN$ via its action on $\cF_P$. This way $\cN$ becomes a torsor under $V_{\cF_P}$ over $D_{\bar z}$ together with a trivialization $\beta_{\cN}$ over $D^*_{\bar z}$. Namely, $\beta_{P,\bar z}$ gives the corresponding trivialization. The extension of scalars of $(\cN, \beta_{\cN})$ via $V\to V/V'$ gives a morphism 
$$
\cG_{\bar z}\to \Gr_{(V/V')_{\cF'_M}, \bar z}
$$
Define $\chi_{\cG}: \cG_{\bar z}\to\Ga$ as the composition
$$
\cG_{\bar z}\to \Gr_{(V/[V,V])_{\cF'_M}, \bar z}\;\toup{\bar\chi_{\bar z}\;\,} \Ga
$$ 

\sssec{} As in Section~\ref{Sect_4.2.6_now}, we have a diagram of projections 
$$
{_{y,\infty}\Bunt_{U, \,{\rm{good\; at}}\; \bar z}^{\cF'_M}}\;\getsup{h^{\la}_{\cG}}\;
 \cG_{\bar z}\;\toup{h^{\ra}_{\cG}}\; {_{y,\infty}\Bunt_{U, \,{\rm{good\; at}}\; \bar z}^{\cF'_M}}
$$
Here $h^{\ra}_{\cG}$ (resp., $h^{\la}_{\cG}$) sends the above point to $(\cF'_G, \kappa'^{\cV})$ (resp., to $(\cF_G, \kappa^{\cV})$). This way $\cG_{\bar z}$ gets a structure of a groupoid over $_{y,\infty}\Bunt_{U, \,{\rm{good\; at}}\; \bar z}^{\cF'_M}$.
 
\sssec{} The rest of the definition of $W$-categories goes through without changes. One first defines
$$
Shv^{W, constr, \heartsuit}_{{\rm{good\; at}}\; \bar z}\subset (Shv(_{y,\infty}\Bunt_{U, \,{\rm{good\; at}}\; \bar z}^{\cF'_M})^{constr})^{\heartsuit}
$$ 
as in Section~\ref{Sect_4.2.7_now}. Then Definitions~\ref{Def_3.2.6}, \ref{Def_3.2.7}, \ref{Def_3.2.8} apply in our situation of $x$ odd giving rise to the categories 
$$
Shv^W(_{y,\infty}\Bunt_U^{\cF'_M}),\;\; Shv^W(_{y,\infty}\Bunt_U^{\cF'_M})^{constr}
$$
with the same formal properties. 
 
  The analog of Proposition~\ref{Pp_3.2.9} holds for $x$ odd with the same proof.
 
\sssec{} Let $\eta=(\cF_M, \beta_M)\in \Gr_M(\cF'_M)$, here $\cF_M$ denotes a $M$-torsor on $X$, and $\beta_M: \cF_M\,\iso\,\cF'_M\mid_{X-y}$ is an isomorphism. 
 
  The notion of positive $s'$-conductor for $\eta$ is defined as in Section~\ref{Sect_3.1.11}. If $\eta$ has a positive $s'$-conductor then one gets a morphism $s: (V/V')_{\cF_M}\to\Omega$ on $X$.  
  
  Recall that $\Bun_U^{\cF_M}$ classifies a $P$-torsor $\cF_P$ together with an isomorphism $\gamma: \cF_P\,\iso\,\cF_M$. We get a natural map $i_{\cF_M}: \Bun_U^{\cF_M}\to {_{y,\infty}\Bunt_U^{\cF'_M}}$, which is equivariant under the above groupoids. So, along the same lines one defines the versions $Shv^W(\Bun_U^{\cF_M})$, $Shv^W(\Bun_U^{\cF_M})^{constr}$. Set 
$$
Shv^W(\Bun_U^{\cF_M})^{constr}=
Shv^W(\Bun_U^{\cF_M})\cap Shv(\Bun_U^{\cF_M})^{constr}
$$  
The functor $i_{\cF_M}^*$ restricts to
$$
i_{\cF_M}^*: Shv^W(_{y,\infty}\Bunt_U^{\cF'_M})\to Shv^W(\Bun_U^{\cF_M})
$$
and preserves the constructibility.   

\sssec{} Let $\Bun_{U/V}^{\cF_M}$ be the stack classifying a $P/V$-torsor $\cF_{P/V}$ on $X$ together with a trivialization $\cF_{P/V}\times_{P/V} M\,\iso\, \cF_M$ on $X$. 
 
  Let $\cX_{\cF_M}$ be the stack classifying exact sequences
\begin{equation}
\label{ex_seq_for_Sect_4.2.5}
0\to (V/V')_{\cF_M}\to ?\to \cO_X\to 0
\end{equation}
on $X$. If $\eta$ has a positive $s'$-conductor then one gets a morphism $\ev_s: \cX_{\cF_M}\to\Ga$ given by the pairing of (\ref{ex_seq_for_Sect_4.2.5}) with $s$.

\sssec{} The natural map $\nu: \Bun_U^{\cF_M}\to \Bun_{U/V}^{\cF_M}$ is surjective on the level of $k$-points. Let $\cW^{\cF_M}$ be the stack classifying a point $(\cF_P,\gamma)\in \Bun_U^{\cF_M}$, and a torsor $\bar\cF$ under $V_{\cF_P}$ on $X$. 
We have a cartesian square
$$
\begin{array}{ccc}
\cW^{\cF_M} & \toup{\act_{\cF_M}} & \Bun_U^{\cF_M}\\
\downarrow\lefteqn{\scriptstyle \pr_1} && \downarrow\lefteqn{\scriptstyle \nu}\\
\Bun_U^{\cF_M} & \toup{\nu} & \Bun_{U/V}^{\cF_M}\\
\end{array}
$$
Here the action map $\act_{\cF_M}$ is defined similarly to the map $\act: \cX_P\to\Bun_P$ from Section~\ref{Sect_1.2.6}, and $\pr_1$ sends the collection $(\cF_P,\gamma, \bar\cF)$ to $(\cF_P, \gamma)$. 

 Let $r_{\cW}: \cW^{\cF_M}\to  \cX_{\cF_M}$ be the map sending a collection $(\cF_P,\gamma, \bar\cF)$ as above to the $(V/V')_{\cF_M}$-torsor $\bar\cF\times_{V_{\cF_P}} (V/V')_{\cF_M}$, which we view as the exact sequence (\ref{ex_seq_for_Sect_4.2.5}). 
 
\begin{Lm} 
\label{Lm_4.2.6}
If $\eta$ has a positive $s'$-conductor then  for any $\cK\in Shv^W(\Bun_U^{\cF_M})^{constr}$ we have an isomorphism
$$
\act_{\cF_M}^*\cK\,\iso\,\pr_1^*\cK\otimes r_{\cW}^*\ev_s^*\cL_{\psi}
$$
Otherwise, $Shv^W(\Bun_U^{\cF_M})^{constr}$ vanishes. 
\end{Lm}
\begin{proof}
This is analogous to (\cite{FGV}, Lemma~6.2.8). 
\end{proof}   

\ssec{End of the proof of Theorem~\ref{Th_preservation_Hecke} for $x$ odd}
\label{Sect_End_proof_x_odd_v2}

\sssec{} Recall that we fixed $\theta\in\Lambda_{G,P}$. Pick a $k$-point $\eta=(\cF_M, \beta_M)\in \Gr_M^{\theta, -}(\cF'_M)$. The fibre of $q_{\theta}$ over $\eta$ identifies with $\Bun_U^{\cF_M}$. As in Section~\ref{Sect_3.1.11}, we get the closed immersion (\ref{inclusion_iota_eta_for_Sect_4.1.13_v2}). It suffices to show that the fibre of (\ref{complex_for_Sect_5.1.6_for_v2}) at $\eta$ vanishes.

\sssec{} Note that $\cE\in Shv^W(\Bun_U^{\cF'_M})$. Set
$$
\cK=\iota_{\eta}^*i_{\theta}^*((j_!\cE)\ast \cA^{\lambda}_G) 
$$
Let $q: \Bun_U^{\cF_M}\to\Bun_G$ be the natura map. It remains to show that
$$
\RG_c(\Bun_U^{\cF_M}, q^*K\otimes \cK)=0
$$
By functoriality, $\cK\in Shv^W(\Bun_U^{\cF_M})^{constr}$. By Lemma~\ref{Lm_4.2.6}, we may and do assume that $\eta$ has a positive $s'$-conductor. It suffices to show that 
$$
\nu_!(q^*K\otimes \cK)=0,
$$
To do so, we will show that $\nu^*\nu_!(q^*K\otimes \cK)=0$. 
 
\sssec{} Pick a $k$-point $(\cF_P^2,\gamma)$ of $\Bun_U^{\cF_M}$. We check that the $*$-fibre of $\nu^*\nu_!(q^*K\otimes \cK)$ at $(\cF_P^2, \gamma)$ vanishes. Let $\cF^2_{P/V'}=\cF_P^2\times_P P/V'$. Since at the generic point of $X$ we have $s=s'$, we get $(\cF_{P/V'}^2, s)\in\cY^0_P$. Our claim follows from Lemma~\ref{Lm_4.2.6} and the fact that $K\in\cF_{\OO}$. Namely, the fibre at $(\cF^2_P,\gamma)$ of 
$$
(\pr_1)_!(r_{\cW}^*\ev_s^*\cL_{\psi}\otimes\act_{\cW}^*q^*K)
$$ 
vanishes by definition of $\cF_{\OO}$.
\QED

\end{document}